\documentclass[final]{siamart251216}

\usepackage{float}
\newfloat{procedure}{htbp}{loa}[section]
\floatname{procedure}{Procedure}

\usepackage{graphicx}
\usepackage{epstopdf} 
\usepackage{hyperref}    
\usepackage{amssymb,latexsym,amsmath}
\usepackage{xcolor}
\usepackage{enumerate}
\usepackage{booktabs}
\usepackage{MnSymbol,bbding,pifont}
\usepackage{needspace}

\usepackage{algorithmic}

\usepackage{tikz, pgfplots, mathtools, tikz-cd}

\pgfplotsset{compat=1.18}
\usetikzlibrary{hobby}
\usetikzlibrary{arrows}
\usetikzlibrary{positioning}
\usetikzlibrary{shapes,snakes}
\usetikzlibrary{cd}

\def\R{\mathbb{R}}
\def\C{\mathbb{C}}

\def\F{\mathbb{F}}
\def\Z{\mathbb{Z}}

\DeclareMathOperator{\rank}{rank}

\DeclareMathOperator{\tr}{trace}

\DeclareMathOperator{\mvec}{vec}

\DeclareMathOperator{\argmin}{argmin}

\newtheorem{remark}[theorem]{{\sc Remark}}
\newtheorem{example}[theorem]{Example}

\newcommand{\hide}[1]{}

\numberwithin{equation}{section}

\begin{document}

\title{Singular spaces for computing the structured distance to singularity} %\thanks{Submitted to the editors on March 1, 2026}}

\author{
    Lauri Nyman\thanks{University of Manchester, Department of Mathematics, M13 9PL, Manchester, England (\email{lauri.nyman@manchester.ac.uk}). Supported by an Engineering and Physical Sciences Research Council (EPSRC) grant EP/Z533786/1.}}

\date{}

\maketitle

\begin{abstract}
    Finding the distance to singularity for a matrix is a ubiquitous problem in numerical linear algebra, and is elegantly solved by the Eckart-Young-Mirsky theorem. Its structured variant arises when one considers structured matrices and wants to preserve their structure. This problem is important for several matrix nearness problems that either entirely or partly reduce to a structured distance to singularity problem. In this work, we propose a new method for the structured distance to singularity problem based on singular spaces, that is, linear subspaces consisting of singular matrices. We study the properties of these spaces to guide the construction of the algorithm. In the Tikhonov regularized formulation, the method globally minimizes a certain objective function alternately in its arguments. We show that a restricted family of singular spaces suffices to attain the best possible improvement in a single iteration.
%In this work, we propose a new method for the structured distance to singularity problem. It is based on globally minimizing a certain objective function alternately in its arguments. The method uses singular spaces, that is, linear subspaces consisting of singular matrices, whose properties we analyze to guide the algorithmic construction.
In many of the numerical comparisons, the proposed method yields substantial reductions in running time when compared with existing methods, while the computed distances to singularity remain comparable. These gains allow larger problems to be solved within practical running times.
% Numerical experiments show substantial reductions in running time compared with existing methods for many of the tested problems, while the computed distances to singularity remain comparable. These gains allow larger problems to be solved within practical running times.
%demonstrate that this new algorithm, in many cases, offers a significant speedup over the state of the art, while the quality of the output remains comparable. %This makes it possible to solve problems of much larger size than what was previously possible. %This speedup is impactful given that existing algorithms are impractically slow for larger problems. 

    % We show that some of the existing work in the literature, including the state of the art, can be understood with the framework of singular spaces. This framework shows that state-of-the-art variable projection approaches are optimal in a precise sense, related to the dimension of the inner optimization problem.

\end{abstract}



\section{Introduction}

In this paper, we consider the problem of finding the structured distance to singularity in the Frobenius norm. This problem can be stated as follows: given a matrix $A \in \F^{n \times n}$, find a singular matrix $B \in \F^{n \times n}$ with a prescribed linear structure that minimizes $\|A-B\|_F$, where $\|\cdot\|_F$ denotes the Frobenius norm and $\F$ is a field of interest, typically $\R$ or $\C$.

This problem arises in several matrix nearness problems involving eigenvalues \cite{Sicilia, imajnapaper, NP21}. For example, finding the smallest structured perturbation $\Delta$ such that $A+\Delta$ has a prescribed eigenvalue $\lambda$ is equivalent to finding the smallest structured perturbation for which $A-\lambda I+\Delta$ is singular. The same idea applies to roots of scalar polynomials through linearizations: prescribing a root corresponds to prescribing an eigenvalue
of a linearization, with perturbations of the polynomial coefficients inducing structured perturbations of the linearization. More generally, eigenvalue problems for matrix polynomials can be studied through structured linearizations of the
form $A-\lambda B$, called matrix pencils \cite{gohberg}. Perturbations of the matrix polynomial coefficients then induce structured perturbations of these pencils, giving rise to structured distance to singularity problems. Related matrix nearness problems for pencils have been studied in \cite{DHpair, NN25}; see \cite{Nyman25thesis} for background.

% While the eigenvalues of a non-monic matrix polynomial can be described in terms of the eigenvalues of the linearizing matrix pencil \cite{DHpair, NN25}. 

% The eigenvalues of matrix polynomials are given by those of its linearization, which in general is a matrix pencil. 

% In the special case when the matrix polynomial is monic, these questions reduce to questions about structued matrices. 

Several matrix nearness problems can be reformulated as structured distance to singularity problems \cite{oracle}. For example, the problem of finding the nearest singular matrix pencil \cite{nearsingpen, andrii, DNN24, glm, KV15}, or more generally, matrix polynomial \cite{bora, DDDG26, ghl, Gnazzo24}, can be seen as a structured variant of the nearest singular matrix problem. %These problems play an important role in understanding the numerical properties of the associated eigenvalue problems \cite{vand81}. 
Other examples include the problem of finding an approximate GCD for two univariate polynomials \cite{oracle, ode_gcd, UseM17gcd, Zeng}, and the problem of finding the distance to instability \cite{oracle, Sicilia}. %These problems are important in many applications, including signal processing, system identification, and analyzing the stability and robustness of linear dynamical systems \cite{gaubitch2005adaptive, agrawal2004common, hinrichsen2005mathematical}.

These problems have both theoretical and practical applications. The distance to singularity of a matrix pencil or matrix polynomial is important for understanding the numerical properties of the associated eigenvalue problems \cite{vand81}. Approximate polynomial GCD problems arise in signal processing and system identification \cite{gaubitch2005adaptive, agrawal2004common}, while
distance to instability is relevant to the stability and robustness of linear dynamical systems \cite{hinrichsen2005mathematical}.

The structured distance to singularity problem is a special case
of structured low-rank approximation, with rank bound $n-1$ for
square matrices of size $n$. Structured low-rank approximation provides a common formulation for problems in system identification, signal processing, and approximate polynomial GCD computation; see, for example, the survey by Markovsky~\cite{Markovsky08}. Broader discussion of matrix nearness problems and their applications can be found in~\cite{guglielmi_book, nicksurvey}.

%The structured distance to singularity problem is a special case of structured low-rank approximation, with rank bound $n-1$ for a square matrix of size $n$. 
Early numerical approaches to structured low-rank approximation
include De Moor's structured total least squares formulation~\cite{DeMoor94}, which seeks a rank-deficient matrix in an affine family by minimizing a quadratic error in the structure parameters. Rosen, Park, and Glick~\cite{RPG} develop a structured total least norm formulation for correcting an overdetermined linear system, which corresponds to imposing rank deficiency on its augmented matrix. Chu, Funderlic, and Plemmons~\cite{ChuFP03} propose numerical methods for structured low-rank approximation with prescribed ranks and linear structures. Subsequent approaches include factorization-based alternating minimization~\cite{IshUM14}, designed for small target ranks, Newton-like iterations for obtaining nearby structured matrices of prescribed rank~\cite{SchS16}, and a semidefinite relaxation of the nearest structured rank-deficient matrix problem~\cite{Cif21}.

% Kernel-based variable projection methods have been developed extensively
% for structured low-rank
% approximation~\cite{MarU13missing,MarU14software,UseM14manifold,UseM14,UseM17gcd},
% with recent work including the Riemann-Oracle framework~\cite{oracle}.
% Applied to the distance-to-singularity problem, these methods express
% singularity through the constraint $(A+\Delta)v=0$, where $\Delta$ is
% the structured perturbation and $v$ is a unit vector. Thus, a single
% vector suffices to represent the rank constraint, even though the
% prescribed rank bound is $n-1$. For a fixed $v$, the perturbation
% variables can be eliminated by solving a linear least-squares problem,
% or a regularised version of it. The resulting reduced objective is
% then minimised numerically over $v$.

Kernel-based variable projection methods have been developed extensively for structured low-rank approximation~\cite{MarU13missing,MarU14software,UseM14manifold,UseM14,UseM17gcd}, with recent work including the Riemann-Oracle framework~\cite{oracle}. %These formulations are particularly relevant when the prescribed rank reduction is small. 
For the structured distance to singularity problem, another
approach uses matrix differential equations within a two-level
iteration~\cite{guglielmi_book,Sicilia}. Related methods for Hankel structured low-rank approximation are studied in~\cite{FazGM21,FazMU25}. More recently, \cite{poloni2026} derives a system of nonlinear equations from the stationary conditions for the structured distance to singularity and solves this system by a Newton iteration. In addition to numerical algorithms, theoretical bounds and, for certain special cases, exact formulas have been obtained in~\cite{Noschese,SharPraj,sieg}.

Despite these developments, efficient  computation of the structured distance to singularity remains challenging for large matrices. Depending on the prescribed structure, existing implementations can
become impractically slow for matrices of order in the thousands. This limits their practical use for large problems arising in
applications and motivates the development of methods with lower computational cost.

% Existing methods for computing the structured distance to singularity can be roughly divided into two categories. The first category consists of variable projection type methods, which Markovsky and Usevich pioneered in their collection of papers \cite{MarU13missing,MarU14software,UseM14manifold,UseM14,UseM17gcd}, while notable recent work includes \cite{oracle}. The second category is characterized by a two-level iteration procedure, and is based on matrix differential equations \cite{guglielmi_book, Sicilia}. {\color{red}Very recent work \cite{poloni2026} uses ideas from both categories and formulates the problem as a system of nonlinear equations. More precisely, this approach uses a clever characterization of the stationary conditions, and solves the resulting system of nonlinear equations via a Newton iteration.} In addition to numerical methods, several papers have derived theoretical bounds, and in some cases exact formulas, for the solution of certain special cases of this problem \cite{Noschese, SharPraj, sieg}. For the general case, no solution formula exists, and existing implementations {\color{red}for most structures }are impractically slow for matrices of order in the thousands. %This makes it practically impossible to address applications that deal with large matrices.

% The variable projection methods \cite{oracle, MarU13missing} can be understood as solving least-squares problems in singular spaces. 
Our approach is based on singular spaces. We call a linear subspace of $\F^{n \times n}$ a singular space if each of its elements is singular. Properties of these spaces have been studied in \cite{eisenbud1988vector, Fillmore, flanders,Lovasz}.
In our approach, we successively solve subproblems in which the sought matrix $B$ is restricted to a fixed singular space. This replaces the singularity constraint by a linear constraint, which leads to least-squares problems that can be solved in closed form. %to access a closed-form solution formula for a constrained version of the original problem. 
Existing kernel-based variable projection methods, such as those in \cite{oracle, MarU13missing}, can also be understood as solving least-squares problems and numerically optimizing over a particular family of singular spaces defined by a common null vector.

%Distance problems with a linear feasible set typically allow for a closed-form solution formula, as this can be achieved by projecting orthogonally to the feasible set. This is what makes singular spaces such a useful tool: they correspond to a linear subset of the feasible set, and allow us to access a closed-form solution formula for a constrained version of the original problem.

In this paper, we develop a theoretical and computational framework for computing the structured distance to singularity based on singular spaces. In the nonregularized setting, the resulting method solves least-squares problems over a sequence of singular spaces. With Tikhonov regularization, we formulate an objective function in two variables: a structured perturbation and a singular space. The method globally minimizes this function alternately in these two arguments.
% With Tikhonov regularization, this approach globally minimizes a certain objective function alternately in its arguments. 
We also develop an augmented Lagrangian variant, in which the singular spaces are parametrized by a unit vector.
% %We propose an alternating minimization approach, in which each update globally minimizes an objective function in one argument while keeping the other fixed.
% The approach we propose is based on globally minimizing a certain objective function alternately in its arguments. %, rather than the variable-projection formulations used in existing methods.
% The resulting algorithm has several desirable qualities when it comes to computational performance, as outlined in Section~\ref{sec:performance}. %Compared to existing methods, this method allows for a broader choice of singular spaces. 
% This new approach yields, in many cases, a very sizeable improvement to the state of the art \cite{oracle, guglielmi_book, poloni2026} in terms of running time, while the outputs remain comparable.
In many of the numerical comparisons, this new approach yields substantial reductions in running time when compared with existing methods \cite{poloni2026, oracle, guglielmi_book}, while the computed distances remain comparable. Although this approach permits more general singular spaces, these gains are already obtained with spaces used in existing variable projection methods.

% {\color{red}This paper is organized as follows. First, in Section~\ref{sec:theory}, we discuss theoretical results that characterize singular spaces. Then, in Section~\ref{sec:problem}, we present a new algorithmic framework for the structured distance to singularity problem. This is followed by numerical experiments in Section~\ref{sec:numerical}. Next, in Section~\ref{sec:performance}, we discuss underlying reasons for the computational speed of the method proposed in this paper. This is followed by Section~\ref{sec:exotic} that discusses the role of uncommon choices for singular spaces in the proposed algorithm. Finally, we present concluding remarks in Section~\ref{sec:conclusion}.

We establish several theoretical properties that support
this framework. %We study the structure of singular spaces and the convergence properties of the resulting algorithms. 
In particular, we prove an optimality property for
singular spaces defined by rectangular zero blocks after
suitable changes of bases. Among all singular spaces containing a prescribed singular
matrix, spaces of this form attain the best possible improvement in the Tikhonov regularized objective function in a single iteration (Theorem~\ref{thm:block_spaces_suffice}). This result motivates the use of zero-block spaces in the proposed method.

% "minimizing spaces" what to do with this? first framework, then minimizer happens to be non-unique and allows for flexibility? or flexibility first? I would like the first framing
For a fixed perturbation, the minimizer of the Tikhonov regularized objective function over singular spaces is not necessarily unique. %We show that, if the objective function is stationary with respect to the perturbation for one of these spaces, it is stationary for all of them
We show that, if the perturbation is stationary for one of these minimizing spaces, it is stationary for all of them, provided that the smallest singular value of the perturbed matrix is simple (Proposition~\ref{prop:stationarity_independent_space}).

Finally, we study the convergence properties of the resulting algorithms. For the inner iteration of the Tikhonov regularized algorithm, we establish that the accumulation points globally minimize the objective function with respect to each argument separately (Proposition~\ref{prop:converge_reg}). Likewise, for the inner iteration of the augmented Lagrangian algorithm, the accumulation points globally minimize the objective function with respect to the perturbation and are stationary with respect to the unit vector (Proposition~\ref{prop:convergence_inner}).

% "minimizing spaces" what to do with this? first framework, then minimizer happens to be non-unique and allows for flexibility? or flexibility first? I would like the first framing
% For the Tikhonov formulation, we also show that stationarity with respect to the structured perturbation is independent of the choice among minimizing singular spaces, provided that the smallest singular value of the perturbed matrix is simple (Proposition~\ref{prop:stationarity_independent_space}).

\subsection*{Notation}
A short summary of notation used throughout this paper is provided in Table \ref{table:notation}. 

\begin{table}[h!]
    \caption{Consolidated notation table.}
    \label{table:notation}
    \begin{center}
        \begin{tabular}{p{0.18\textwidth}p{0.75\textwidth}}
            \toprule
            \textbf{Notation} & \textbf{Description} \\
            \midrule
            $\F$ & A field of interest, typically $\R$ or $\C$. \\
            $\mathcal{S}_n(\F)$ & The set of singular matrices of size $n \times n$ over $\F$. \\
            $\mathfrak{S}_n(\F)$ & The set of all singular spaces over $\F$. \\
            $\mathfrak{S}_n(\F;M)$ & The set of all singular spaces over $\F$ containing the matrix $M$. \\
            $\mathcal{S}_v, \mathcal{S}^*_v$ & The singular spaces $\{ B \in \F^{n\times n} : B v = 0\}$ and $\{ B \in \F^{n\times n} : B^* v = 0\}$, respectively. \\
            $\mathcal{T}$ & A set of structured matrices, $\mathcal{T} \subset \F^{n\times n}$. \\
            $M_{I,J}$ & The submatrix of $M$ given by row indices $I$ and column indices $J$. \\
            $M_I$ & The submatrix formed by the columns of $M$ given by the index set $I$. \\
            $M^{-*}$ & The conjugate transpose of the inverse of $M$. \\
            $\mbox{Proj}_S(M)$ & The orthogonal projection of a matrix $M$ onto a linear subspace $S$. \\
            $\mvec(M)$ & The vectorization of a matrix $M$. \\
            $M^\dagger$ & The Moore--Penrose pseudoinverse of $M$. \\
            $e_{i,j}$ & The standard basis matrix with $1$ in position $(i,j)$ and $0$ elsewhere. \\
            \bottomrule
        \end{tabular}
    \end{center}
\end{table}
% }
% This leads to a computational method that, as outlined in Section \ref{sec:regularization}, globally minimizes a certain objective function alternately in its arguments. The resulting algorithm has several desirable qualities, as outlined in Section~\ref{sec:performance}.

% In this paper, we propose a new method for finding the structured distance to singularity, based on the concept of singular spaces. As opposed to existing methods, this method is based on a systematic formulation in terms of singular spaces. that treats singular spaces as the central organizing object. 

% Compared to existing methods, 

% In this paper, we perform a systematic formulation of the structured distance to singularity problem in terms of singular spaces.

% The novelty of this work lies in the general 

% The novelty of this method Compared to existing approaches, }

% We analyze the properties of singular spaces in relation to the structured distance to singularity problem, and propose a practical algorithm based on a certain subfamily of singular spaces. The resulting algorithm, outlined in Section \ref{sec:regularization}, globally minimizes a certain objective function alternately in its arguments. 

\section{Linear structures inside the set of singular matrices} \label{sec:theory}

Let $n \geq 2$ and let $\mathcal{S}_n(\F)$ denote the set of singular matrices of size $n \times n$ over the field $\F$. In other words, $\mathcal{S}_n(\F) := \{ A \in \F^{n \times n} \ : \ \det A = 0 \}$. In this writing, the field $\F$ is chosen to be either the field of real numbers $\R$ or the field of complex numbers $\C$. In the following, we omit the field whenever it is clear from the context, that is, we write $\mathcal{S}_n$ in place of $\mathcal{S}_n(\F)$. 

In this section, we study the linear structure inside $\mathcal{S}_n(\F)$. Section~\ref{sec:motivation} introduces singular spaces through examples and relates them to existing variable projection methods. Section~\ref{sec:singvec} then develops theoretical results on their dimensions and structure.

% In this section, we focus on characterizing the linear structure inside $\mathcal{S}_n(\F)$. Section \ref{sec:motivation} gives a brief exposition and motivation for the study of the linear structure inside $\mathcal{S}_n(\F)$, while in Section \ref{sec:singvec}, we focus on characterizing singular spaces.

\subsection{Linear geometry in $\mathcal{S}_n(\F)$} \label{sec:motivation}
%As a warm-up, let us consider the following characterization for vector subspaces, the proof of which is left as a simple exercise.
%\begin{proposition}\label{proposition:vec_space}
% A set $S \subset \F^{n}$ is a vector space if and only if for every pair of points $(x,y) \in S^2$ there exists an $\F$-affine subspace $A_{x,y}$ of $S$ such that {\color{red}$x,y \in A_{x,y}$.}
% \end{proposition}

%As $\mathcal{S}_n(\F)$ is not a vector space, not every pair of singular matrices can be connected with a single subspace that is included in $\mathcal{S}_n(\F)$. However, every pair of singular matrices can be connected via a sequence of two subspaces (rather trivially, as $\mathcal{S}_n(\F)$ is closed under scalar multiplication; for $A,B\in \mathcal{S}_n(\F)$, it holds that $\{tA \ | \ t \in \F\}$ and $\{tB \ | \ t \in \F\}$ intersect at origin). This property holds because $\mathcal{S}_n(\F)$ has one of the two main properties of a vector space: it is closed under scalar multiplication. The other main property, being closed under addition, does not hold in general. For example, $e_1 e_1^T, e_2 e_2^T \in \F^{2\times 2}$ are clearly singular, but their sum is the identity matrix, which is clearly not singular. A subset of singular matrices for which both main properties of a vector space hold is called a \textit{singular space}. %In other words, singular spaces are vector subspaces inside $\mathcal{S}_n(\F)$.

The set of singular matrices $\mathcal{S}_n(\F)$ possesses one of the two main properties of a vector space: it is closed under scalar multiplication. The other main property, being closed under addition, does not hold in general. For example, $e_1 e_1^T, e_2 e_2^T \in \F^{2\times 2}$ are clearly singular, but their sum is the identity matrix, which is clearly not singular. A subset of singular matrices for which both main properties of a vector space hold is called a \textit{singular space}.
\begin{definition}
    A matrix subset $S \subset \mathcal{S}_n(\mathbb{F})$ is called a singular space if it is a vector space, that is, for all $v,w \in S$ and $\alpha,\beta \in \F$, it holds that $\alpha v + \beta w \in S$.
\end{definition}

The following example shows that there exist high dimensional singular spaces. In fact, every element of $\mathcal{S}_n(\F)$ belongs to such a space.
% \begin{example} \label{example:highdim}
%     Let $A \in \mathcal{S}_n(\F)$. Let $V \in \F^{n \times (n-1)}$ denote a full-rank matrix such that $\mbox{im} (A^*) \subset \mbox{im} (V)$. Then, the set $S:=\{BV^* : B \in \F^{n \times (n-1)}\}$ is an $(n^2 - n)$-dimensional vector space over $\F$ such that $A \in S$.
% \end{example}
\begin{example} \label{example:highdim}
    % % Let $A \in \mathcal{S}_n(\F)$. The sets $S:=\{B \in \F^{n \times n} : \mbox{im}(B) \subset \mbox{im}(A)\}$ and $S^*:=\{B \in \F^{n \times n} : \mbox{im}(B) \subset \mbox{im}(A)\}$ are singular spaces of dimension $n \dim \mbox{im}(A)$. 
    % Let $A \in \mathcal{S}_n(\F)$, and let $V \in \F^{n \times (n-1)}$ denote a full-rank matrix such that $\mbox{im} (A^*) \subset \mbox{im} (V)$. Then, the set $S:=\{C \in \F^{n \times n} : \mbox{im}(C^*) \subset \mbox{im}(V)\} = \{BV^* : B \in \F^{n \times (n-1)}\}$ is an $(n^2 - n)$-dimensional vector space over $\F$ such that $A \in S$.
    Let $V \in \F^{n \times (n-1)}$ have full column rank.
    The set
    \[
        S := \{BV^* : B \in \F^{n \times (n-1)}\}
    \]
    is a linear subspace of $\F^{n \times n}$. Every element
    of $S$ has rank at most $n-1$, so $S$ is a singular space.
    Moreover, since $V$ has full column rank, the linear map
    $B \mapsto BV^*$ is injective. Hence, $S$ has dimension
    $n(n-1)$.

    Every singular matrix $A \in \mathcal{S}_n(\F)$ belongs
    to a space of this form. Indeed, since
    $\rank A \leq n-1$, we can choose a full-rank matrix
    $V \in \F^{n \times (n-1)}$ such that
    $\mbox{im}(A^*) \subseteq \mbox{im}(V)$.
    This implies that $A = BV^*$ for some
    $B \in \F^{n \times (n-1)}$, and hence $A \in S$.
\end{example}
Example \ref{example:highdim} highlights that much of the geometry of the set $\mathcal{S}_n(\F)$ is in fact linear. In particular, the singular space $S$ of Example \ref{example:highdim} has codimension $n$, and hence its ratio with the dimension of the ambient space $\F^{n \times n}$ goes to zero as $n$ goes to infinity. %Therefore, in this precise relative sense, $\mathcal{S}_n(\F)$ approaches a linear set in the limit $n \rightarrow \infty$.

The singular spaces in Example~\ref{example:highdim} can also be defined more succinctly in terms of a shared vector in the kernel.

%State-of-the-art methods for finding the structured distance to singularity, that are based on variable projection \cite{oracle, MarU13missing}, can be understood in the framework of singular spaces. In particular, they rely on singular spaces of Example \ref{example:highdim}, by using the formulation given in Example \ref{ex:same_kernel} below. This example defines these type of singular spaces more succinctly in terms of a shared vector in the kernel.
\begin{example} \label{ex:same_kernel}
    Let $v \in \F^n \backslash \{0\}$ and let $\mathcal{S}_v = \{ B \in \F^{n\times n} : B v = 0\}$ and $\mathcal{S}^*_v = \{ B \in \F^{n\times n} : B^* v = 0\}$. It holds that the sets $\mathcal{S}_v$ and $\mathcal{S}^*_v$ are singular spaces. The spaces $\mathcal{S}_v$ and $\mathcal{S}_v^*$ are related by conjugate transposition:
    \[
        \mathcal{S}_v^* = \{B^* : B \in \mathcal{S}_v\}.
    \]
    Both spaces have dimension $n(n-1)$.
\end{example}
It is not difficult to see that the singular space $\mathcal{S}_v$ of Example \ref{ex:same_kernel} coincides with the set $\{BV^* : B \in \F^{n \times (n-1)}\}$ defined in Example \ref{example:highdim}, with $v$ being orthogonal to $\mbox{im}(V)$. These spaces underlie many existing variable projection methods, including those in \cite{oracle, MarU13missing}. The framework we develop in this paper also allows using spaces with a different structure from $\mathcal{S}_v$.
%This paper aims to provide a general framework for the use of singular spaces in the structured distance to singularity problem, allowing for spaces with a different structure from $\mathcal{S}_v$. With this goal in mind, Section \ref{sec:singvec} examines the structure of singular spaces from a theoretical point of view, while Section \ref{sec:problem} proposes a practical algorithm based on this analysis for the structured distance to singularity problem.

%The aim of Section \ref{sec:singvec} is to examine the structure of singular spaces, while keeping in mind the connection with matrix nearness problems. 

\Needspace{3\baselineskip}
\subsection{Structure of singular spaces} \label{sec:singvec}

\subsubsection{Dimension bounds and shared null vectors}
In Example~\ref{ex:same_kernel}, the matrices in $\mathcal{S}_v$
share the right null vector $v$, while those in $\mathcal{S}_v^*$
share the left null vector $v$. The following example shows that
a singular space does not need to have a common nonzero vector in either
kernel.
% We start by noting that not all elements of a singular space necessarily share the same non-zero vector in the left or right kernel. This is demonstrated by the following example.
\begin{example} \label{ex:not_same_kernel}
    Consider the set
    \begin{align*}
        S = \left\{ \
        \begin{bmatrix}
            a_{1} & a_{2} & a_{3} \\
            0     & 0     & a_{4} \\
            0     & 0     & a_{5}
        \end{bmatrix} \ : \ a_i \in \F \ \right\}.
    \end{align*}
    It holds that $S$ is a singular space with $$\bigcap_{A \in S}\mbox{ker} A = \bigcap_{A \in S}\mbox{ker} A^* = \{0\}.$$
\end{example}
The following theorem by Flanders shows that singular spaces, whose elements do not share a nontrivial null vector, cannot exceed the dimension of $\mathcal{S}_v$ and $\mathcal{S}^*_v$ of Example~\ref{ex:same_kernel}.
\begin{theorem}[\cite{flanders}]\label{prop:dim}
    Let $S \subset \F^{n\times n}$ be a vector space of singular matrices over the field $\F$. It holds that $\dim S \leq n(n-1)$.
\end{theorem}

Proposition \ref{prop:dim2} shows that any singular space with dimension more than $n^2-2n+2$ is necessarily a subset of $\mathcal{S}_v$ or $\mathcal{S}^*_v$ for some $v$.
\begin{proposition}[\cite{Fillmore}]\label{prop:dim2}
    Let $S \subset \F^{n\times n}$ be a singular space over the field $\F$ with dimension $d > n^2-2n+2$. It holds that all elements of $S$ either share a common nonzero vector in the left kernel or in the right kernel.
\end{proposition}

\subsubsection{Spaces characterized by zero patterns}

We next study singular spaces that can be characterized by zero patterns after suitable changes of bases. These spaces are important for the algorithm we outline in Section \ref{sec:problem} for two reasons. First, such spaces are easy to construct so that they include a prescribed singular matrix. Second, among all singular spaces containing the same prescribed matrix, the smallest value of the objective function can always be attained by such a space. These results are stated more precisely in Procedure \ref{proc:construct} and Theorem~\ref{thm:block_spaces_suffice}.
% In the Tikhonov regularized method, we choose a singular space
% containing a fixed nearest singular matrix. We will show in
% Theorem~\ref{thm:block_spaces_suffice} that, among these choices,
% a zero-block space attains the smallest objective value after
% the subsequent minimization over the perturbation.

% Next, we aim to characterize singular spaces that have a nice structure from an algorithmic point of view. 

Let $e_{i,j}$ denote the matrix with a $1$ in position $(i,j)$ and zeros elsewhere. %that has exactly one non-zero element which is in position $(i,j)$ and is equal to 1. 
In the following, we assume that our singular space admits a basis $\mathcal{B} \subset \{e_{i,j} \}_{i,j}$, that is, $\mathcal{B}$ is a subset of the standard basis. Results proved under this assumption naturally extend to all vector spaces whose bases can be transformed into a subset of the standard basis via a rank-preserving linear transformation\footnote{Rank-preserving linear transformations $\F^{n\times n} \rightarrow \F^{n\times n}$ are either of form $M \mapsto X M Y^*$ or $M \mapsto X M^T Y^*$, where $X,Y \in \mbox{GL}(n, \F)$, as originally proved in \cite[Theorem 1]{Marcus_Moyls_1959} for the field of complex numbers, but the result holds more generally for any algebraically closed field with characteristic zero.}. 
% As we will see in Section \ref{sec:newmethod}, %this assumption on the basis is useful in designing a computational algorithm. In particular, 
% for a given singular matrix $A$, it is easy to construct singular spaces of this form that include $A$.

\begin{remark}
    A matrix subspace that admits a basis $\mathcal{B} \subset \{e_{i,j} \}_{i,j}$ is entirely characterized by a particular pattern of zeros, and has no further structure: the element in position $(i,j)$ must be zero whenever $e_{i,j}\notin\mathcal{B}$, while all remaining entries can vary independently.% This zero pattern is determined by the elements of the standard basis that do not appear in $\mathcal{B}$.
\end{remark}

A useful case is when the prescribed zero pattern includes a rectangular
block whose side lengths sum to $n+1$. The following proposition
shows that this guarantees singularity, also after invertible
changes of bases.

From here on, the notation $M_{I,J}$ stands for a submatrix of $M$ that is formed by keeping the rows and columns given by the index sets $I$ and $J$, respectively. Moreover, we denote by $M^{-*}$ the conjugate transpose of $M^{-1}$.
\begin{proposition} \label{prop:blocks_are_singular}
    Let $S$ be a vector subspace of $\F^{n \times n}$ for which there exist $X, Y \in \mbox{GL}(n, \F)$ and index sets $I$ and $J$ with $|I| + |J| = n+1$ such that $(X^{-1}M Y^{-*})_{I,J}=0$ for all $M \in S$. Then, $S$ is a singular space.
\end{proposition}
\begin{proof}
    Let $E_I \in \F^{n \times |I|}$ be a matrix $[e_{f(1)} \ \cdots \ e_{f(|I|)}]$, where $f$ is a bijection between $\{1, 2, \dots, |I| \}$ and $I$. In other words, the columns of $E_I$ are the standard basis vectors given by the index set $I$. Let us define the matrix $E_J$ similarly. Then, $E_I^T (X^{-1}M Y^{-*}) E_J = 0$. It follows from standard rank inequalities that $$0 = \mbox{rank} \ E_I^T (X^{-1}M Y^{-*}) E_J \geq \mbox{rank} \ E_I^T + \mbox{rank} \ (X^{-1}M Y^{-*}) +  \mbox{rank} \ E_J - 2n.$$ Rearranging the terms and noting that $|I| + |J| = n+1$ yields $n-1 \geq \mbox{rank} \ M$.
\end{proof}
%Proposition \ref{prop:blocks_are_singular} essentially states that permutations of a matrix subspace of block-triangular matrices with a sufficiently large zero block are necessarily singular spaces, which is a rather intuitive result. 
We next show that the converse holds for matrix subspaces entirely determined by a zero pattern: if such a subspace is
singular, its zero pattern must contain a rectangular zero
submatrix whose side lengths sum to $n+1$. %a singular matrix subspace that is entirely characterized by a pattern of zeros must be of block triangular form (up to permutations) with a sufficiently large zero block.
%The reverse direction of Proposition \ref{prop:blocks_are_singular} holds for vector spaces with bases that can can be transformed into a subset of the standard basis via a rank-preserving linear transformation. 
This is a corollary of a more general result stated in Proposition \ref{prop:singular_are_block}. Before stating the proposition, we introduce a useful lemma.
\begin{lemma} \label{lemma:rankminors}
    Let $X, Y \in \mbox{GL}(n, \F)$ and let $X_i$ and $Y_i$ denote the $i$th column of the matrices $X$ and $Y$, respectively. Let $S$ be a singular space with a basis $\mathcal{B} \subset \{X_i Y_j^* \}_{i,j}$. Then, $$\max_{M \in S} \rank M = \min \{k \in \Z : |\{X_i Y_{\sigma(i)}^*\}_i \cap \mathcal{B}| \leq k \ \forall \sigma \in \mbox{Sym}(n)\}.$$
\end{lemma}
\begin{proof}
    We start by noting that the statement is equivalent with $$\max_{M \in X^{-1} S Y^{-*}} \rank M = \min \{k \in \Z : |\{e_{i,\sigma(i)}\}_i \cap X^{-1}\mathcal{B}Y^{-*}| \leq k \ \forall \sigma \in \mbox{Sym}(n)\}.$$ Here, $X^{-1}\mathcal{B}Y^{-*}$ is a basis of $X^{-1}S Y^{-*}$, and is also a subset of the standard basis. For brevity, let us define $\widetilde{\mathcal{B}}:= X^{-1}\mathcal{B}Y^{-*}$, $\widetilde{{S}}:= X^{-1}{S}Y^{-*}$ and $$r:=\min \{k \in \Z : |\{e_{i,\sigma(i)}\}_i \cap X^{-1}\mathcal{B}Y^{-*}| \leq k \ \forall \sigma \in \mbox{Sym}(n)\}.$$ Let $M \in \widetilde{S}$ and let $a_{ij}$ denote the coefficients of $M$ in the standard basis, that is, $M = \sum a_{ij} e_{i,j}$. Note that $e_{i,j} \notin \widetilde{\mathcal{B}}$ implies that $a_{ij}=0$. With this observation, it is clear that minors of $M$ of higher order than $r$ need to vanish. Hence, $\rank M \leq r \ \forall M \in \widetilde{S}$, and hence $\max_{M \in \widetilde{S}} \rank M \leq r$. To reach the inequality in the other direction, consider that if $\max_{M \in \widetilde{S}} \rank M < r$, then all $r$ minors of all elements of $\widetilde{S}$ would vanish, which would imply that $r$ is not the smallest number satisfying the constraint, which is a contradiction.
\end{proof}
% The following result follows from Lemma~\ref{lemma:rankminors} and K\"onig's theorem.
\begin{proposition} \label{prop:singular_are_block}
    Let $X, Y \in \mbox{GL}(n, \F)$ and let $X_i$ and $Y_i$ denote the $i$th column of the matrices $X$ and $Y$, respectively. Let $S$ be a vector space with a basis $\mathcal{B} \subset \{X_i Y_j^* \}_{i,j}$ such that $\max_{M \in S} \rank M = r$. Then, there exist index sets $I$ and $J$ with $|I| + |J| = 2n-r$ such that $M \in S$ implies that $(X^{-1}M Y^{-*})_{I,J}=0$.
\end{proposition}
\begin{proof}
    Let $P=\{(i,j): e_{i,j}\in X^{-1}\mathcal B Y^{-*}\}$. By Lemma~\ref{lemma:rankminors}, the maximum number of positions in $P$ lying in distinct rows and columns is $r$. % After the change of bases, the allowed positions are those $(i,j)$ for which $e_{i,j}\in X^{-1}\mathcal B Y^{-*}$.
    % By Lemma~\ref{lemma:rankminors}, the maximum number of allowed positions in distinct rows and columns is $r$.
    By K\"onig's theorem, in the matrix formulation stated
    in~\cite[p.~126]{Brualdi1966}, there exist sets of row indices $R$ and column indices $C$ with $|R|+|C|=r$ such that every $(i,j)\in P$ satisfies $i\in R$ or $j\in C$. For $I=\{1,\ldots,n\}\setminus R$ and $J=\{1,\ldots,n\}\setminus C$, it then holds that $(X^{-1}MY^{-*})_{I,J}=0$ for every $M\in S$.
    Moreover, $|I|+|J|=2n-r$.
\end{proof}

We note that the singular space of Example \ref{ex:not_same_kernel} is a special case of Proposition \ref{prop:singular_are_block} with $n=3, r=2$ and $|I| = |J| = 2$. Moreover, we note that choices with $r=n-1$, $|I|=2, |J|=n-1$ (or equivalently $|I|=n-1, |J|=2$) yield $n^2-2n+2$ dimensional singular spaces with no shared nonzero vector in the left kernel or the right kernel. In light of Proposition \ref{prop:dim2}, these spaces have maximal dimension among singular spaces that are not mere subsets of spaces of form $\mathcal{S}_v$ or $\mathcal{S}^*_v$ of Example \ref{ex:same_kernel}. The spaces $\mathcal{S}_v$ and $\mathcal{S}^*_v$ themselves correspond to the cases $|J| = 1$ and $|I|=1$ of Proposition \ref{prop:singular_are_block}, respectively, with $r=n-1$.

The result of Proposition \ref{prop:singular_are_block} does not hold in general for all singular spaces: the assumption that the space is entirely characterized by its underlying zero pattern with no further structure (i.e., it has a basis $\mathcal{B} \subset \{X_i Y_j^* \}_{i,j}$) is needed. Example~\ref{ex:counter} below shows one singular space for which the result of Proposition \ref{prop:singular_are_block} does not hold. Note that the space used in this example has more structure than its underlying zero pattern: for example, the elements in positions (1,1) and (2,3) need to always be equal.
\begin{example} \label{ex:counter}
    Consider the set
    \begin{align*}
        S = \left\{ \
        \begin{bmatrix}
            y  & x & 0 \\
            0  & z & y \\
            -z & 0 & x
        \end{bmatrix} \ : \ x,y,z \in \F \ \right\}.
    \end{align*}
    It holds that $S$ is a singular space for which the result in Proposition \ref{prop:singular_are_block} does not hold, that is, for all $A, B \in \mbox{GL}(3, \F)$ and for all index sets $I$ and $J$ with $|I| + |J| = 4$ there exists $M \in S$ such that $(A^{-1}MB^{-1})_{I,J} \neq 0$.
\end{example}
\begin{proof}
    The proof is given in Appendix~\ref{appendix:counter}.
\end{proof}

Every singular space of dimension at most two, however, is contained in a zero-block singular space, as shown by the following lemma. We will use this result in Theorem~\ref{thm:block_spaces_suffice} to show that these spaces suffice to attain the smallest value of the objective function after the next perturbation update.

\begin{lemma} \label{lemma:two_dimensional_blocks}
    Let $S \subset \F^{n\times n}$ be a singular space with $\dim S \leq 2$. Then, after orthogonal (if $\F=\R$) or unitary (if $\F=\C$) changes of bases, $S$ is contained in a singular space defined by a rectangular zero block whose side lengths sum to $n+1$.
\end{lemma}

\begin{proof}
    Write $S=\operatorname{span}_{\F}\{M,N\}$. Since $S$ is singular, we have $\det(M+\lambda N)\equiv0$ for all $\lambda \in \F$. By the generalized Schur decomposition \cite{Stewart1972}, there exist unitary matrices $Q,Z \in \C^{n\times n}$ such that $R:=QMZ$ and $T:=QNZ$ are upper triangular. When $\F=\R$, the singularity of the pencil allows $Q,Z$ to be chosen real and orthogonal, with $R,T$ still upper triangular; see \cite[Lemma~6.1]{DNN24}.

    Since $M+\lambda N$ is singular, the determinant is zero for all values for $\lambda$. This implies that $\prod_{j=1}^{n}(r_{jj}+\lambda t_{jj}) \equiv 0.$
    This can only hold if at least one of the factors is zero. Let $k$ denote the index of such a factor, that is, $r_{kk}=t_{kk}=0$. Upper triangularity then implies that $R_{I,J}=T_{I,J}=0$, where $I=\{k,\ldots,n\}$, $J=\{1,\ldots,k\}$. As such, $M$ and $N$ belong to the set
    \[
        C=\left\{
            X \in \F^{n\times n} :
            (QXZ)_{I,J}=0
        \right\},
    \]
    which is a singular space by Proposition~\ref{prop:blocks_are_singular}, since $|I|+|J|=n+1$.
\end{proof}

\subsubsection{Singular spaces through a fixed matrix}
Later in Section \ref{sec:newmethod}, we will describe a method for computing the structured distance to singularity. As we will see, this algorithm is only aware of singular spaces that include the current point of the optimization procedure. In the following, we prove a result suggesting that this restriction may exclude part of the local geometry of $\mathcal{S}_n$. %we prove a result that suggests that such an algorithm would be oblivious to some of the local geometry of $\mathcal{S}_n$. 
Although the proof is valid only for the case $\F = \R$, we conjecture that a similar result holds also in the case $\F = \C$. %This motivates an alternative approach that will be described in Section \ref{sec:regularization}.

\begin{proposition} \label{prop:dimensions_missing}
    Let $M \in \mathcal{S}_n(\R)$. There exists an $n$-dimensional affine subspace $A$ of $\R^{n\times n}$ such that $M \in A$ and $A\cap S=\{M\}$ for every singular space $S$ %$A$ intersects trivially every singular space $S$ 
    satisfying $M \in S$ and $\max \{ \rank C : C \in S\} = \rank M$.
\end{proposition}

\begin{proof}
    The proof is given in Appendix~\ref{appendix:dimensions_missing}.
\end{proof}

Although every singular matrix belongs to a singular space, restricting to singular spaces that contain a fixed matrix $M$ may exclude singular matrices arbitrarily close to $M$. Proposition~\ref{prop:dimensions_missing} shows that this is the case for real matrices $M$ of rank $n-1$, for which the rank assumption is automatically satisfied. This suggests that it can be useful to relax the requirement that the next singular space contains the current matrix, as we do in Section~\ref{sec:regularization}.

\subsubsection{Compactness of the set of singular spaces}
We denote the set of all singular spaces by
\begin{align*}
    \mathfrak{S}_n(\F) := \{V \subset \mathcal{S}_n(\F): V \mbox{ is a vector space over } \F \}.
\end{align*}

\begin{remark} \label{prop:compact}
    The set $\mathfrak{S}_n(\F)$ is compact in a natural way. The proof of this can be found in Appendix~\ref{appendix:compact}.
\end{remark}

\section{Structured distance to singularity} \label{sec:problem}
In this section, we consider the problem of finding the structured distance to singularity. This section is organized as follows. First, in Section \ref{sec:problem_statement}, we consider the problem statement. Then, in Section \ref{sec:newmethod}, we devise a computational approach based on singular spaces. In Sections \ref{sec:regularization} and \ref{sec:lagrangian}, we outline two distinct ways to regularize the approach of Section \ref{sec:newmethod}. Finally, in Section \ref{sec:convergence}, we discuss convergence properties of the proposed method. %We propose a computational approach that relies on singular spaces, and utilizes the analysis done in Section \ref{sec:singvec}. 

%Singular spaces are a useful tool for this problem because it is easy to solve distance problems over linear spaces. Indeed, the solution to the problem $\min_{B\in V} \|A-B\|,$ where $V$ is a vector subspace, and $\|\cdot\|$ is a norm given by an inner product, is simply given by the orthogonal projection of $A$ onto the subspace $V$. As such, there exists a closed-form expression for the distance to any given singular space. This remains true even when the singular space is intersected with a linear structure, since the intersection of two linear spaces is necessarily linear. The high-level idea is then to construct a solution to the original problem from the solutions to these linearly constrained subproblems. The description of this process, in the context of our method, is made more precise later in this section.

% Good paragraph. Maybe can find some use for it
%In a broader context, solving a part of an optimization problem exactly, and numerically optimizing over the remaining part, is called variable projection. For this problem, variable projection type methods \cite{oracle, MarU13missing} often lead to an ill-behaved objective function that can be discontinuous, even at the minimizer (see \cite{oracle}). 

\subsection{Problem statement} \label{sec:problem_statement}
The problem of finding the structured distance to singularity can be defined with respect to any matrix norm. In this paper, we use the Frobenius norm and consider the constrained minimization problem
\begin{align} \label{eq:problem}
    \min_{\Delta \in \mathcal{T}} \| \Delta \|_F \quad \mbox{s.t.} \quad A+\Delta \mbox{ is singular},
\end{align}
where $\mathcal{T} \subset \F^{n\times n}$ is the set of structured matrices of interest.\footnote{Although we focus on square matrices in the theoretical
analysis, the problem could also be defined for rectangular matrices, in which case the quality of being singular would be replaced with the quality of being rank-deficient. We illustrate the application of the method to such a problem in Section~\ref{sec:NLEVP}.} %While it is possible to extend the theoretical analysis of earlier sections to the rectangular setting, for simplicity, we consider the square case in this paper. This also simplifies the resulting algorithm presented in this section. } %In the following, we set $\| \cdot \|=\| \cdot \|_F$. 
In this paper, we focus on linear structures, that is, we require that $\mathcal{T}$ is a linear subspace of $\F^{n\times n}$. 
% The above problem statement can be written more compactly as
% \begin{align} 
%     \min_{\Delta \in \mathcal{T} \cap (\mathcal{S}-A)} \| \Delta \|_F.
% \end{align}
In practice, it is often the case that $A \in \mathcal{T}$, and we make this assumption throughout this paper for simplicity of exposition. In principle, the proposed method could be extended to problems with $A \not \in \mathcal{T}$ as well.

% REMOVE THIS?
% We first note that finding the minimizer for the minimization problem \eqref{eq:problem} would be easy if the constraint were an affine set: in this case, the minimizer would simply be the orthogonal projection of the zero matrix to the feasible set. However, the constraint in \eqref{eq:problem} is nonlinear due to the nonlinearity of the set $\mathcal{S}$. 

% -----------------------------
% The placement of this paragraph is weird. The phrasing is weird here in general. This  section also explains overview for our approach.
% What is left to do is to construct a solution to the original problem from these solutions. In a regularized formulation, we can find, for a fixed perturbation $\Delta$, a subspace that minimizes the objective. We use this as a more intelligent way to find the optimal $\Delta_*$.
% For the numerical solution of \eqref{eq:problem}, the approach that we take in this paper attempts to exploit the high-dimensional linear structure of the set $\mathcal{S}$. The general idea is to first exctract linear spaces from $\mathcal{S}$ that approximate the set well, and then solve the problem over these spaces, which can be done cheaply and exactly. A regularized formulation allows for solving the reverse problem: for a fixed perturbation $\Delta$, find a subspace that minimizes the objective. Combining these two processes together leads to an efficient algorithm that will be explained in Section \ref{sec:regularization}. Before that, in Section \ref{sec:newmethod}, we look at the vanilla, non-regularized approach to the problem. 

\subsection{Formulation using singular spaces} \label{sec:newmethod}
%For the numerical solution of \eqref{eq:problem}, the general idea is to restrict the set $\mathcal{S}$ to one of its linear subsets, which lets us solve the problem cheaply and exactly. A regularized formulation allows for solving the reverse problem: for a fixed perturbation $\Delta$, find a subspace that minimizes the objective. Combining these two processes together leads to an efficient algorithm that will be explained in Section \ref{sec:regularization}. Before that, in this section, we look at the vanilla, non-regularized approach to the problem. 

% In the regularized case, out method becomes an alternating minimization type approach, as opposed to a variable projection approach.

Let us begin by observing that $$\mathcal{S}_n
    = \bigcup_{S \in \mathfrak{S}_n(\F)} S,$$ which follows from the fact that every singular matrix belongs to some singular space. We can therefore express \eqref{eq:problem} as %a minimization problem over a structured perturbation $\Delta$ and a singular space $S$:
\begin{equation} \label{eq:problem_joint}
    \min_{\substack{\Delta \in \mathcal{T}\\
                    S \in \mathfrak{S}_n(\F)}}
    \|\Delta\|_F
    \quad \mbox{s.t.} \quad A+\Delta \in S.
\end{equation}

For fixed $S$, the constraints $\Delta \in \mathcal{T}$ and
$A+\Delta \in S$ define the affine feasible set
$\mathcal{T} \cap (S-A)$. The minimizer has a closed-form solution formula given by an orthogonal projection
\[
    \Delta_S = \mbox{Proj}_{\mathcal{T} \cap S}(A)-A,
\]
where $\mbox{Proj}_{\mathcal{T} \cap S}$ denotes the orthogonal projection to the space $\mathcal{T} \cap S$.
We then define a function $f : \mathfrak{S}_n(\F) \rightarrow \R$ such that
\begin{align} \label{eq:obj_fun}
    f(S)= \min_{\Delta \in \mathcal{T} \cap (S-A)} \| \Delta \|_F,
\end{align}
and what is left to do is to solve $\min_{S \in \mathfrak{S}_n(\F)} f(S)$. %The set $\mathfrak{S}_n(\F)$ is not a Riemannian manifold, so we cannot directly rely on the gradient or the Hessian of the function $f$ to exploit its local geometry when computing $\min_{S \in \mathfrak{S}_n(\F)} f(S)$. 
The geometry of $\mathfrak{S}_n(\F)$ makes it nontrivial to apply standard optimization methods based on gradients or Hessians to this remaining minimization. Instead, we find a cheap way to generate a sequence of points $(S_{i})_{i \in \Z^+}$ with $\{S_{i}\}_{i \in \Z^+} \subset \mathfrak{S}_n$ such that $f(S_{i+1})\leq f(S_{i})$. Note that, if it holds that $A + \Delta_S \in T \in \mathfrak{S}_n$ for some $T \neq S$, it then follows that $f(T)\leq f(S)$. As such, one approach for attempting to minimize the function $f$ is generating a sequence of singular subspaces $(S_{i})_{i \in \Z^+}$ such that $A + \Delta_{S_i} \in S_{i+1}$. This approach is outlined in Algorithm \ref{alg:alg1}.

\begin{remark} 
    Variable projection methods, including those in \cite{oracle, MarU13missing}, use a reformulation of the problem similar to \eqref{eq:obj_fun}. These approaches use the fact that a matrix is singular if and only if there is a non-zero vector in the kernel. This implies that $\mathcal{S}_n = \bigcup_{v \in \F^n \backslash \{0\} }\mathcal{S}_v$, where $\mathcal{S}_v$ is defined as in Example~\ref{ex:same_kernel}. These methods then define $f(v):= \min_{\Delta \in \mathcal{T} \cap (\mathcal{S}_v-A)} \| \Delta \|_F$, and solve $\min_{\|v\|=1} f(v)$ via a first or second order numerical minimization scheme. In other words, these approaches use singular spaces of the form $\mathcal{S}_v$ and exploit the smooth structure of the unit sphere $\{v \in \F^n :\|v\|=1\}$.
\end{remark}

\begin{algorithm}
    \caption{Nonregularized approach}
    \label{alg:alg1}
    \begin{algorithmic}
        \STATE{\textbf{Input:} tolerance $\mbox{tol}$}
        \STATE{Define $\Delta_{1} := -A$,  $\Delta_{0} := \infty I$, $i:=1$}
        \WHILE{$ \|\Delta_{i-1} - \Delta_{i}\|_F > \mbox{tol}$}
        \STATE{Update $i := i+1$}
        \STATE{Find $S_i \in \mathfrak{S}_n$ s.t. $\Delta_{{i-1}} \in S_{i} - A$}
        \STATE{Define $\Delta_{{i}} := \arg\min_{\Delta \in \mathcal{T} \cap ({S}_{i}-A)} \| \Delta \|_F$}
        \ENDWHILE
        \RETURN $\Delta_i$
    \end{algorithmic}
\end{algorithm}

\begin{proposition} \label{prop:terminate}
    Algorithm \ref{alg:alg1} terminates after finitely many iterations for every $\mbox{tol}>0$.
\end{proposition}
% \begin{proof}
%     Let $(\Delta_i)_{i\in\Z^+}$ be the sequence of solutions as defined in Algorithm \ref{alg:alg1}. As $\Delta_i$ always satisfies the constraint of the minimization problem associated with $\Delta_{i+1}$, it must hold that $\|\Delta_{i+1}\|_F \leq \|\Delta_{i}\|_F$. Since this sequence is non-increasing and is bounded from below by zero, it must converge.
% \end{proof}
\begin{proof}
    Let $(\Delta_i)_{i\in\Z^+}$ be the sequence of solutions as defined in Algorithm \ref{alg:alg1}. Since $\Delta_{i+1}$ and $\Delta_i$ are both elements of the same affine set $\mathcal{T}\cap(S_{i+1}-A)$ and $\Delta_{i+1}$ is its minimum-norm element, it follows from orthogonality that
    \[
        \|\Delta_i\|_F^2 = \|\Delta_{i+1}\|_F^2 + \|\Delta_i-\Delta_{i+1}\|_F^2.
    \]
    The sequence $(\|\Delta_i\|_F^2)_{i\in\Z^+}$ is non-increasing
    and bounded below by zero, and hence converges. Therefore,
    $\|\Delta_i-\Delta_{i+1}\|_F \to 0$, which proves finite
    termination for every $\mbox{tol}>0$.
\end{proof}

To implement Algorithm~\ref{alg:alg1}, we need to construct the next singular space $S_{i}$ such that $\Delta_{{i-1}} \in S_{i} - A$. This can be achieved by using singular spaces that admit a basis $\mathcal{B} \subset \{e_{i,j} \}_{i,j}$ up to rank-preserving linear transformations. In the transformed coordinates, these spaces are defined entirely by prescribed zero entries, while all remaining entries vary independently. As per Proposition \ref{prop:singular_are_block}, the zero pattern of each of these subspaces contains a zero submatrix of size $|I| \times |J|$, where $|I| + |J| = n+1$. Let us express $A+\Delta_{i-1}$ in terms of its singular value decomposition $U \Sigma V^*$. Since $\Sigma$ is diagonal and has at least one zero diagonal entry, there exist zero submatrices of all possible sizes in $\Sigma$ such that the side lengths add up to $n+1$. As such, $A+\Delta_{i-1}$ is contained in the corresponding singular spaces, and we can use any of these for the next iteration. This procedure is stated in Procedure \ref{proc:construct}. %It is implicit in Procedure \ref{proc:construct} that the choice for the set $S_i$ should be different from $S_{i-1}$, otherwise Algorithm \ref{alg:alg1} necessarily terminates.

\begin{procedure}
    \caption{Constructing a singular space} \label{proc:construct}
    \centering
    \begin{tabular}{p{0.9\textwidth}}
        \hline
        \textbf{Step 1.} Express $A+\Delta_{i-1}$ in terms of its singular value decomposition $U \Sigma V^*$. \\
        \textbf{Step 2.} Choose a submatrix $\Sigma_{I,J} = 0$ such that $|I| + |J| = n+1$.                    \\
        \textbf{Step 3.} Set $S_i := \operatorname{span}\{U_k V_j^*\}_{(k,j) \notin I \times J}$.                           \\
        \hline
    \end{tabular}
\end{procedure}
Let $\{B^{(i)}\}_{1 \leq i \leq p}$ denote an orthonormal basis of the structure space $\mathcal{T} \subset \F^{n \times n}$ with respect to the Frobenius inner product ${\langle A, B \rangle = \tr (A^* B)}$. Throughout this paper, we let $\delta\in\F^p$ denote the coordinate vector of $\Delta\in\mathcal{T}$ in this basis. In other words,
\[
    \Delta=\sum_{j=1}^p B^{(j)}\delta_j.
\]
\begin{proposition} \label{thm:fv}
    Let ${S}_{i} \in \mathfrak{S}_n(\F)$ and $A \in \mathcal{T}$. Define
    \begin{equation} \label{Mr}
        M = \begin{bmatrix}
            \mvec{(\mbox{Proj}_{S_{i}^\perp}} B^{(1)}) & \mvec{(\mbox{Proj}_{S_{i}^\perp}} B^{(2)}) & \dots & \mvec{(\mbox{Proj}_{S_{i}^\perp}} B^{(p)})
        \end{bmatrix} \in \mathbb{F}^{n^2 \times p}.
    \end{equation}
    Then, the solution to the minimization problem in~\eqref{eq:obj_fun} is unique and is given by
    \[
        \Delta_* = \sum_{j=1}^p B^{(j)} (\delta_*)_j, \quad  \delta_* = - M^\dagger \mvec{(\mbox{Proj}_{S_{i}^\perp} A)},
    \]
    where $M^\dagger$ denotes the Moore--Penrose pseudoinverse. Hence, the function $f$ is well-defined with $f(S_i) = \|M^\dagger \mvec{(\mbox{Proj}_{S_{i}^\perp} A)}\|$.
\end{proposition}
\begin{proof}
    Note that, because of orthonormality of the basis, $\|\Delta_*\|_F = \|\delta_*\|$. Expressing the constraint in vectorized form yields
    \begin{align}
        \min_{\Delta \in \mathcal{T} \cap (\mathcal{S}_i-A)} \| \Delta \|_F \quad = \quad \min_{\delta \in \F^p} \| \delta \|_F \quad \mbox{s.t.} \quad M \delta = - \mvec{(\mbox{Proj}_{S_{i}^\perp} A)}.
    \end{align}
    A solution exists because $A \in \mathcal{T}$; namely, the choice for $\delta$ corresponding to $-A$ satisfies the constraint. The minimal norm solution to an underdetermined least-squares problem can be obtained by left-multiplying the right hand side with the Moore--Penrose pseudoinverse \cite[Theorem~1.2.10]{Bjorck}.
\end{proof}

The result of the following proposition, Proposition \ref{prop:fast_eval}, shows that the size of the system matrix $M$ can be reduced from $n^2 \times p$ to $|I||J| \times p$, where $|I|+|J|=n+1$. %This lets us evaluate the expression for $\delta_*$ in fewer operations. 
We note that the matrices $X,Y$ in the statement of the proposition correspond to the matrices $U$ and $V$ in Procedure \ref{proc:construct}, respectively. In the following, we write $A_I$ to denote a matrix formed by taking the columns of the matrix $A$ given by the index set~$I$.
\begin{proposition} \label{prop:fast_eval}
    Let $\delta_*$ be defined as in Proposition \ref{thm:fv}, and let $S$ be as in Proposition \ref{prop:singular_are_block} with $X,Y$ unitary and $r=n-1$ and $\dim S = n^2 - |I||J|$. It holds that

    \[
        \delta_* = - \widetilde M^\dagger \mvec{((X_I)^*A Y_J)},
    \]
    where
    \begin{equation*}
        \widetilde M = \begin{bmatrix}
            \mvec{((X_I)^* B^{(1)} Y_J)} & \mvec{((X_I)^* B^{(2)} Y_J)} & \dots & \mvec{((X_I)^* B^{(p)} Y_J)}
        \end{bmatrix} \in \mathbb{F}^{|I||J| \times p}.
    \end{equation*}

\end{proposition}
\begin{proof}
    Under the assumptions of the proposition,
    \[
        S = \{B \in \F^{n\times n} : (X_I)^* B Y_J = 0\}.
    \]
    Given that $\Delta=\sum_{j=1}^p B^{(j)}\delta_j$, the constraint $A+\Delta\in S$ is equivalent with
    \[
        \widetilde M\delta
        = -\mvec{((X_I)^* A Y_J)}.
    \]
    Pseudoinverse yields the minimum-norm solution, which gives the result.
    % We have that
    % \begin{align*}
    %     \min_{\Delta \in \mathcal{T} \cap (\mathcal{S}_i-A)} \| \Delta \|_F = \min_{\Delta \in X^* (\mathcal{T} \cap (\mathcal{S}_i-A)) Y} \| X \Delta Y^* \|_F \quad =  \min_{\Delta \in (X^* \mathcal{T} Y) \cap (X^* \mathcal{S}_i Y-X^*AY)} \| \Delta \|_F,
    % \end{align*}
    % where the fact that the Frobenius norm is unitarily invariant was used in the second step. This corresponds to the problem
    % \begin{align*}
    %     \quad \min_{\delta \in \F^p} \| \delta \|_F \quad \mbox{s.t.} \quad M_{XY} \delta = - \mvec{(\mbox{Proj}_{(X^* S_{i} Y)^\perp} X^* A Y)},
    % \end{align*}
    % where
    % \begin{equation*}
    %     M_{XY} = \begin{bmatrix}
    %         \mvec{(\mbox{Proj}_{(X^* S_{i} Y)^\perp}} (X^* B^{(1)} Y)) & \dots & \mvec{(\mbox{Proj}_{(X^* S_{i} Y)^\perp}} (X^* B^{(p)} Y))
    %     \end{bmatrix}.
    % \end{equation*}
    % As it holds for all $D\in \F^{n\times n}$ that $$(\mbox{Proj}_{(X^* S_{i} Y)^\perp} D)_{i,j} = 0, \quad \forall (i,j) \notin I \times J,$$
    % both sides are zero for these indices and we can restrict to the subsystem corresponding to the indices $(i,j) \in I \times J$. For these, it holds that $$(\mbox{Proj}_{(X^* S_{i} Y)^\perp} D)_{I,J} = D_{I,J},$$
    % which implies the solution
    % \[
    %     \delta_* = - \widetilde M^\dagger \mvec{((X^*A Y)_{I,J})}.
    % \]
\end{proof}
% In order to avoid the computation of a full \texttt{svd} in every iteration, the formula given in Proposition \ref{prop:fast_eval} can be further simplified for the cases $|I| = 1$ and $|J| = 1$. In the following corollary, the case $|J|=1$ recovers the formula in \cite[Theorem~2.4]{oracle}, while $|I|=1$ gives the corresponding formula for a left null vector.

The following corollary gives the standard minimum-norm formulas used in kernel-based variable projection; see, for example, \cite[Section~3.1]{UseM14} and \cite[Theorem~2.4]{oracle}. We state them in the notation of Proposition~\ref{prop:fast_eval}.
\begin{corollary} \label{cor:fastest_eval}
    If $|I| = 1$ in Proposition \ref{prop:fast_eval}, then

    \[
        \delta_* = - M_X^\dagger \mvec{((X_I)^*A)},
    \]
    where
    \begin{equation*}
        M_X = \begin{bmatrix}
            \mvec{((X_I)^* B^{(1)})} & \mvec{((X_I)^* B^{(2)})} & \dots & \mvec{((X_I)^* B^{(p)})}
        \end{bmatrix} \in \mathbb{F}^{n \times p}.
    \end{equation*}
    Likewise, if $|J| = 1$, then

    \[
        \delta_* = - M_Y^\dagger \mvec{(A Y_J)},
    \]
    where
    \begin{equation*}
        M_Y = \begin{bmatrix}
            \mvec{(B^{(1)} Y_J)} & \mvec{(B^{(2)} Y_J)} & \dots & \mvec{(B^{(p)} Y_J)}
        \end{bmatrix} \in \mathbb{F}^{n \times p}.
    \end{equation*}

\end{corollary}
\begin{proof}
    Similar to the proof of Proposition \ref{prop:fast_eval} upon setting $Y$ to be the identity matrix in the first case, and $X$ to be the identity matrix in the second case.
\end{proof}
Corollary \ref{cor:fastest_eval} shows that, in order to compute $\delta_*$, we only need one left singular vector in the case $|I|=1$, and one right singular vector in the case $|J|=1$, both corresponding to the smallest singular value of $A+\Delta_{j-1}$ of Procedure \ref{proc:construct}. Recall that the cases $|I|=1$ and $|J|=1$ correspond to singular spaces of the form $\mathcal{S}_v^*$ and $\mathcal{S}_v$, respectively. As such, in terms of cost per iteration, these singular spaces are good choices for Algorithm~\ref{alg:alg1}.

In the numerical experiments, we compare these spaces with spaces defined by rectangular zero blocks of sizes $2\times(n-1)$ or $(n-1)\times2$. %As alternatives, we consider spaces defined by rectangular zero blocks of sizes $2\times(n-1)$ or $(n-1)\times2$. 
These latter spaces have no common nonzero left or right null vector, and have the highest dimension among singular spaces with this property, in light of Propositions \ref{prop:dim2} and \ref{prop:singular_are_block}.

Although Proposition~\ref{prop:terminate} shows that the algorithm terminates after finitely many iterations, the computed perturbation is not necessarily optimal. In particular, choosing the same singular space in two successive iterations gives the same minimum-norm solution, and the algorithm terminates, even if another space would allow further improvement. %Moreover, Proposition~\ref{prop:dimensions_missing} suggests that requiring the singular spaces to contain the current matrix may limit the possible improvements. This motivates the Tikhonov regularized formulation in Section~\ref{sec:regularization}, which relaxes this requirement and chooses the next singular space by globally minimizing the regularized objective with the perturbation fixed.
The same can happen for distinct singular spaces whose intersections with the structure space $\mathcal{T}$ coincide, as the following example shows. 

% The choice of singular spaces in Algorithm~\ref{alg:alg1} must also take the structure $\mathcal{T}$ into account. Example \ref{ex:symmetry} below shows that distinct singular spaces can give the same feasible set and hence the same minimizer.
 %have the same intersection with the structure $\mathcal{T}$. In this case, choosing $\mathcal{S}_v^*$ after $\mathcal{S}_v$ gives the same feasible set and hence the same minimizer.

% does not change the feasible set of the subproblem, which motivates considering singular spaces of other forms.
% This suggests that it can be beneficial to be flexible in what choices for singular spaces are allowed in the algorithm.
\begin{example} \label{ex:symmetry}
    Let $\mathcal{T} \subseteq \{M \in \R^{n \times n} : M=M^T \}$ be a linear space. Then, $\mathcal{T} \cap \mathcal{S}_v = \mathcal{T} \cap \mathcal{S}^*_v$ for all $v \in \R^n$. Hence, for this choice of $\mathcal{T}$, using $\mathcal{S}_v$ and $\mathcal{S}^*_v$ in successive iterates in Algorithm \ref{alg:alg1} gives the same minimizer and causes the algorithm to terminate. %would not lead to a meaningful approach for the structure $\mathcal{T}$.
\end{example}

% These have no common nonzero
% left or right null vector and attain the maximal dimension
% among singular spaces with this property, by
% Propositions~\ref{prop:dim2} and~\ref{prop:singular_are_block}.

% Good paragraph. Maybe can be used somewhere, e.g. in conclusion?
% The idea is to search for the nearest singular matrix by moving along singular spaces. singular spaces cover most of the dimensions in the set, and moving along them is computationally cheap. Note that ideally we would like the highest possible dimension for the singular spaces, so as to maximize how well they approximate Sing. 

\subsection{Tikhonov regularization} \label{sec:regularization}
In Algorithm \ref{alg:alg1}, the next singular space ${S}_{{i}}$ is chosen from those that include $A+\Delta_{i-1}$. Proposition~\ref{prop:dimensions_missing} implies that this approach will inevitably exclude some of the local geometry of $\mathcal{S}_n$ at $A+\Delta_{i-1}$. As such, it can be beneficial to relax the condition $\Delta_{i} \in {S}_{{i}} - A$ with regularization techniques. The benefit of regularization techniques for matrix nearness algorithms has been observed in \cite{oracle} as well. %There, the authors noted that variable projection methods \cite{oracle, MarU13missing} sometimes lead to an ill-behaved objective function that can be discontinuous, even at the minimizer.

We relax the constraint $A+\Delta \in S$ in \eqref{eq:problem_joint} by incorporating it in the objective function as a penalty term. In this way, we obtain the regularized problem
\begin{align} \label{eq:problem_g}
    \min_{\substack{\Delta \in \mathcal{T}\\
                   S \in \mathfrak{S}_n(\F)}}
    \|\Delta\|_F^2
    + \frac{1}{\varepsilon}
      \|\mbox{Proj}_{S^\perp}(A+\Delta)\|_F^2,
\end{align}
where $\varepsilon>0$ is a regularization parameter. We define the regularized objective function as
% Writing $\Delta=\sum_{j=1}^p B^{(j)}\delta_j$ in the orthonormal basis of $\mathcal{T}$, we define
\begin{align} \label{eq:obj_g}
    g_\varepsilon(\delta,S)
    = \|\delta\|^2
    + \frac{1}{\varepsilon}\|M_S\delta-r_S\|^2,
\end{align}
where $M_S\delta=r_S$ expresses the constraint $A+\Delta\in S$, as in
Proposition~\ref{thm:fv}. We solve the regularized problem by alternately minimizing $g_\varepsilon(\delta,S)$ with respect to $\delta$ and $S$. %, keeping the other variable fixed.

We call this formulation the \textit{Tikhonov regularized problem}. The idea is to successively solve for $\min_{\delta,S} g_\varepsilon(\delta,S)$ for decreasing values for $\varepsilon$, while using the minimizer $\delta_*$ of the previous iteration as the starting point for the next iteration. Under suitable assumptions, the sequence of these solutions converges to the solution of the original problem \eqref{eq:problem} in the limit $\varepsilon \rightarrow 0$. This statement is made more precise in Section~\ref{sec:convergence} by using the more general theory of Riemannian augmented Lagrangian methods presented in \cite{LiuBoumal}.

% For fixed $S$, we define $f_\varepsilon(S)$ to be the minimum value with respect to the perturbation:
% \begin{align} \label{eq:regMr}
%     f_\varepsilon(S) = \quad \min_{\Delta \in \mathcal{T}} \| \Delta \|^2_F + \frac{1}{\varepsilon} \|\mbox{Proj}_{S^\perp} (A+\Delta) \|_F^2.
% \end{align}

% Relaxing the constraint $\Delta \in (S_i-A)$ in \eqref{eq:obj_fun} can be done by incorporating it in the objective function as a penalty term. That is, we want to find
% \begin{align} \label{eq:regMr}
%     f_\varepsilon(S_i) = \quad \min_{\Delta \in \mathcal{T}} \| \Delta \|^2_F + \frac{1}{\varepsilon} \|\mbox{Proj}_{S_{i}^\perp} (A+\Delta) \|_F^2,
% \end{align}
% where $\varepsilon$ is a regularization parameter. We call this formulation the \textit{Tikhonov regularized problem}. The idea is to successively solve for $\min_S f_\varepsilon(S)$ for decreasing values for $\varepsilon$, while using the minimizer $S_*$ of the previous iteration as the starting point for the next iteration. Under suitable assumptions, the sequence of these solutions converges to the solution of the original problem \eqref{eq:problem} in the limit $\varepsilon \rightarrow 0$. This statement is made more precise in Section \ref{sec:convergence} by using the more general theory of Riemannian augmented Lagrangian methods presented in \cite{LiuBoumal}.
\begin{proposition} \label{prop:regularized_minimum_exists}
    Let $g_\varepsilon$ be as defined in \eqref{eq:obj_g}. The problem 
    \[
        \min_{\substack{\delta\in\F^p\\
                       S\in\mathfrak{S}_n}}
        g_\varepsilon(\delta,S)
    \]
    % $$\min_{S \in \mathfrak{S}_n} f_\varepsilon(S)$$ 
    is well-defined, that is, there exists a minimizer of the function $g_\varepsilon$.
\end{proposition}
\begin{proof}
    The result follows from the extreme value theorem, after noting that (i) the minimization can be restricted to a compact set, and (ii) $g_\varepsilon$ is jointly continuous, where both (i) and (ii) hold with respect to the same metric, the product metric induced by the Euclidean norm on $\F^p$ and the metric $d$ on $\mathfrak{S}_n$ defined in Appendix~\ref{appendix:compact}. 
    To see why (i) and (ii) are true, let us consider the function $$(S, \Delta) \mapsto \| \Delta \|^2_F + \frac{1}{\varepsilon} \|\mbox{Proj}_{S^\perp} (A+\Delta) \|_F^2.$$

    To see why (i) is true, note that the feasible set $\mathfrak{S}_n(\F)$ is compact (as per Remark \ref{prop:compact}), and that the minimum value of $g_\varepsilon$ over $\delta\in\F^p$ is bounded from above by the value attained at $\Delta = -A$, which is $\|A\|_F^2$. Since $g_\varepsilon(\delta,S)\geq\|\delta\|^2$, the minimization can be restricted to the product of two compact sets, $K:=\{\delta\in\F^p:\|\delta\|^2\leq\|A\|_F^2\}$ and $\mathfrak{S}_n(\F)$, which is itself compact.
    
    To see (ii), note that $(S, \Delta) \mapsto \| \Delta \|^2_F + \frac{1}{\varepsilon} \|\mbox{Proj}_{S^\perp} (A+\Delta) \|_F^2$ is jointly continuous, as the orthogonal projection matrix onto a subspace depends continuously on the subspace with respect to the metric $d|_{\mathfrak{S}_n(\F)}$. Since $\Delta=\sum_{j=1}^p B^{(j)}\delta_j$ depends continuously on $\delta$, this also establishes the joint continuity of $g_\varepsilon$. %Since $K$ is a compact set that does not depend on $S$, Berge's maximum theorem implies that the minimum value function $f_\varepsilon(S)$ is continuous.}
\end{proof}
% We have seen that the minimization problem \eqref{eq:obj_fun} can be expressed as
% \begin{align*}
%     \quad \min_{\delta \in \F^p} \| \delta \|^2 \quad \mbox{s.t.} \quad M_S \delta = r_S,
% \end{align*}
% where the system matrix $M_S$ and the right-hand side $r_S$ depend on the choice for the singular space $S$ (as per Proposition \ref{thm:fv}, Proposition \ref{prop:fast_eval} and Corollary \ref{cor:fastest_eval}). In the regularized formulation corresponding to \eqref{eq:regMr}, we have
% \begin{align} \label{eq:regMr2}
%     f_\varepsilon(S) = \quad \min_{\delta \in \F^p} \| \delta \|^2 + \frac{1}{\varepsilon} \| M_S \delta - r_S \|^2.
% \end{align}
% The solution to this is given by
% \begin{equation} \label{deltastar}
%     \delta_* = M_S^*(M_S M_S^*+\varepsilon I)^{-1}r_S = (M_S^*M_S+\varepsilon I)^{-1}M_S^*r_S,
% \end{equation}
% with
% \begin{align*}
%     f_\varepsilon(S) = r_S^*(M_S M_S ^* + \varepsilon I)^{-1}r_S,
% \end{align*}
% see \cite[Theorem 2.9.]{oracle}.

% The final question we need to address is how to update the singular space $S_{i}$ in this regularized formulation. Note that Procedure \ref{proc:construct} does not directly apply, since the matrix $A+\Delta_{j-1}$ is not necessarily singular. For this purpose, let us define
% \begin{align} \label{eq:obj_g}
%     g_\varepsilon(\delta,S) = \quad \| \delta \|^2 + \frac{1}{\varepsilon} \|M_S \delta - r_S \|^2,
% \end{align}
% so that
% \begin{align} \label{eq:problem_g}
%     \min_{S \in \mathfrak{S}_n} f_\varepsilon(S) = \min_{S \in \mathfrak{S}_n} \min_{\delta \in \F^{p}} g_\varepsilon(\delta, S).
% \end{align}
For fixed $S$, setting the gradient of $g_\varepsilon(\delta,S)$ with respect to $\delta$ to zero gives the unique minimizer
\begin{equation} \label{deltastar}
    \delta_*
    = M_S^*(M_S M_S^*+\varepsilon I)^{-1}r_S
    = (M_S^*M_S+\varepsilon I)^{-1}M_S^*r_S.
\end{equation}
%For the case $S=\mathcal{S}_v$, see \cite[Theorem~2.9]{oracle}. 
We can also minimize $g_\varepsilon(\delta,S)$ with respect to $S$ for fixed $\delta$. %The final question we need to address is how to update the singular space $S_{i}$ in this regularized formulation. 
For fixed $\delta$, a space $S\in\mathfrak{S}_n$ is a global minimizer of $g_\varepsilon(\delta,S)$ if and only if $S$ contains a nearest singular matrix to $A+\Delta$ in the Frobenius norm. We can therefore use Procedure~\ref{proc:construct} to construct a singular space containing such a matrix, which minimizes $g_\varepsilon(\delta,S)$ over $S\in\mathfrak{S}_n$. This process is outlined in Procedure \ref{proc:construct_reg}.

% With a small modification to Procedure \ref{proc:construct}, we can find a new subspace $S_{i}$ that minimizes $g_\varepsilon(\delta, S)$ over $S \in \mathfrak{S}_n$ for a fixed $\delta$. To see how to do this, let us again express $A+\Delta_{i-1}$ in terms of its singular value decomposition $U \Sigma V^*$. It is clear that there exists a $|I| \times |J|$ submatrix of $\Sigma$ for all possible values for $|I|,|J|$ satisfying $|I| + |J| = n+1$ that contains $\sigma_{min}$ as the only possibly non-zero element. The distance to the singular space defined by such a $I \times J$ submatrix (see Proposition \ref{prop:blocks_are_singular}) is clearly $\sigma_{\min}$, which is also the distance to the set of singular matrices, and is hence the global minimum over all singular spaces. This process is outlined in Procedure \ref{proc:construct_reg}.

\begin{procedure}
    \caption{Constructing $\argmin_{S \in \mathfrak{S}_n} g_\varepsilon(\delta_{i-1}, S)$} \label{proc:construct_reg}
    \centering
    \begin{tabular}{p{0.9\textwidth}}
        \hline
        \textbf{Step 1.} Express $A+\Delta_{i-1}$ in terms of its singular value decomposition $U \Sigma V^*$.                  \\
        \textbf{Step 2.} Choose a submatrix $\Sigma_{I,J}$ such that $|I| + |J| = n+1$ and $\|\Sigma_{I,J}\|_F = \sigma_{min}$. \\
        \textbf{Step 3.} Set $S_i := \operatorname{span}\{U_k V_j^*\}_{(k,j) \notin I \times J}$.                                             \\
        \hline
    \end{tabular}
\end{procedure}

Procedure \ref{proc:construct_reg} suggests a strategy for solving the problem \eqref{eq:problem_g}: we can alternate between the minimization of the function \eqref{eq:obj_g} along the  first argument $\delta$ and the second argument $S$. %This kind of approach, where the function is alternately minimized along different variables, is called \emph{block coordinate descent}. The word ``block'' refers to the fact that the minimum is found for a block of coordinates simultaneously: in our case, we have two blocks, where the first block consists of $p$ coordinates corresponding to the first argument $\delta$ and the second ``block" consists of the second argument $S$ which is minimized over the set $\mathfrak{S}_n(\F)$. 
The resulting algorithm is outlined in Algorithm \ref{alg:reg}.

For a fixed $\delta$, there are generally many minimizers for the singular space $S$. In Procedure~\ref{proc:construct_reg}, different choices for minimizing spaces arise from choosing different sizes for the index sets $I$ and $J$. The following proposition shows that, if the smallest singular value of $A+\Delta$ is simple, stationarity with respect to $\delta$ is independent of the choice of a minimizing space.

\begin{proposition}\label{prop:stationarity_independent_space}
    Suppose that the smallest singular value of $A+\Delta$ is simple. Then,
    $\nabla_\delta g_\varepsilon(\delta,S)$ is independent of the choice of
    \[
        S\in\operatorname*{argmin}_{T\in\mathfrak{S}_n(\F)}
        g_\varepsilon(\delta,T).
    \]
    In particular, if $\delta$ is a stationary point of $g_\varepsilon(\cdot,S)$ for one such space, it is a stationary point for all of them. In this case, for every minimizing space $S$, the pair $(\delta,S)$ globally minimizes $g_\varepsilon$ with respect to each argument separately.
\end{proposition}
\begin{proof}
    Since the smallest singular value of
    $A+\Delta$ is simple, there exists a unique nearest singular
    matrix $\widehat A$ to $A+\Delta$ in the Frobenius norm.
    For every space $S$ minimizing $g_\varepsilon(\delta,S)$,
    the matrix $\operatorname{Proj}_S(A+\Delta)$ is a nearest
    singular matrix to $A+\Delta$. Hence, $\operatorname{Proj}_S(A+\Delta)=\widehat A$ and $\operatorname{Proj}_{S^\perp}(A+\Delta)=A+\Delta-\widehat A$.

    Let $h\in\F^p$ and $H=\sum_{j=1}^p B^{(j)}h_j$. The directional derivative of $g_\varepsilon$ with respect to $\delta$ in the direction $h$ is
    \[
        \begin{aligned}
            D_\delta g_\varepsilon(\delta,S)[h]
            &=
            2\operatorname{Re}(\delta^*h)
            +\frac{2}{\varepsilon}
            \operatorname{Re}
            \left\langle
                \operatorname{Proj}_{S^\perp}(A+\Delta),
                \operatorname{Proj}_{S^\perp}(H)
            \right\rangle_F \\
            &=
            2\operatorname{Re}(\delta^*h)
            +\frac{2}{\varepsilon}
            \operatorname{Re}
            \langle A+\Delta-\widehat A,H\rangle_F,
        \end{aligned}
    \]
    where the second equality follows from the self-adjointness of the orthogonal projection. This expression is independent of $S$, which proves
    the result. %In the complex case, derivatives are understood with respect to the underlying real Euclidean structure.
    % The final statement follows because $g_\varepsilon(\delta,T)$ is strictly convex in $\delta$ for each fixed $T$, and $T$ already minimizes the objective with $\delta$ fixed.
\end{proof}

If $\delta$ is not stationary, different choices of a minimizing space $S$ can still give different values of the objective function $g_\varepsilon$ after the subsequent minimization over $\delta$. The following theorem shows that, among all singular spaces containing the same nearest singular matrix, the smallest such value of the objective function is attained by a singular space defined solely by a zero-block constraint as in Proposition~\ref{prop:blocks_are_singular}. This motivates restricting the choice of singular spaces to zero-block spaces in Procedure~\ref{proc:construct_reg}.

For a singular matrix $M$, let
\[
    \mathfrak{S}_n(\F;M)
    :=\{S\in\mathfrak{S}_n(\F):M\in S\}
\]
denote the set of all singular spaces containing $M$.

\begin{theorem} \label{thm:block_spaces_suffice}
    Let $M\in\mathcal{S}_n(\F)$. There exists a space $\widetilde S\in\mathfrak{S}_n(\F;M)$, defined solely by a zero-block constraint as in Proposition~\ref{prop:blocks_are_singular}, such that
\[
    \min_{\delta\in\F^p} g_\varepsilon(\delta,\widetilde S)
    =
    \min_{S\in\mathfrak{S}_n(\F;M)}
    \min_{\delta\in\F^p} g_\varepsilon(\delta,S).
\]
\end{theorem}

\begin{proof}
The right-hand side is well-defined: since $\mathfrak{S}_n(\F;M)$ is a closed subset of $\mathfrak{S}_n(\F)$, it is compact and the proof of Proposition~\ref{prop:regularized_minimum_exists} can be used to prove the existence of a minimizer of $g_\varepsilon$ over $\F^p\times\mathfrak{S}_n(\F;M)$.

Let $(\delta_*,S_*)$ denote such a minimizer. Define
    \[
        A_*=A+\sum_{j=1}^p B^{(j)}(\delta_*)_j,
        \quad
        \widehat A_*=\operatorname{Proj}_{S_*}(A_*).
    \]
    Since $M,\widehat A_*\in S_*$, the space $\operatorname{span}_{\F}\{M,\widehat A_*\}$ is singular and has dimension at most two. By Lemma~\ref{lemma:two_dimensional_blocks}, it is contained in a zero-block singular space $\widetilde S$ defined as in Proposition~\ref{prop:blocks_are_singular}. %In particular, $M,\widehat A_*\in\widetilde S$. 
    Since $\widehat A_*\in\widetilde S$, we have
    \[
        \|\operatorname{Proj}_{\widetilde S^\perp}(A_*)\|_F
        \leq \|A_*-\widehat A_*\|_F
        =\|\operatorname{Proj}_{S_*^\perp}(A_*)\|_F.
    \]
    It therefore holds that
    \[
        \begin{aligned}
            \min_{\delta\in\F^p}g_\varepsilon(\delta,\widetilde S)
            \leq g_\varepsilon(\delta_*,S_*).
        \end{aligned}
    \]
    % Since this also holds for a space $S\in\mathfrak{S}_n(\F;M)$ that minimizes $\min_{\delta\in\F^p}g_\varepsilon(\delta,S)$, the result follows. 
    This proves the result.
    %whose existence follows by the same compactness argument as above, proves the claim.
    % To establish attainment of the minimum, note that
    % $\mathfrak{S}_n(\F;M)$ is a nonempty closed subset of
    % the compact set $\mathfrak{S}_n(\F)$.
    % Moreover, since $A\in\mathcal{T}$, the perturbation
    % $\Delta=-A$ gives the objective value $\|A\|_F^2$
    % for every $S$. Together with
    % $g_\varepsilon(\delta,S)\geq\|\delta\|^2$, this shows
    % that the minimization over $\delta$ can be restricted
    % to the compact ball $\|\delta\|\leq\|A\|_F$.
    % Joint continuity of $g_\varepsilon$ therefore guarantees
    % the existence of a minimizing pair $(\delta,S)$.
    % Applying the first part of the proof to a minimizing
    % space $S$ gives a zero-block space attaining the same
    % minimum value.
\end{proof}

\begin{algorithm}
    \caption{Tikhonov regularization approach}
    \label{alg:reg}
    \begin{algorithmic}
        \STATE{\textbf{Input:} initial value for $\varepsilon$, tolerances $\mbox{tol}_1, \mbox{tol}_2,$ \\
        \quad a decrease parameter $k$, and a matrix $A = \sum_{l=1}^p B^{(l)} (\alpha)_l,$ \\
        \quad where $\{B^{(l)}\}_{l=1}^p$ is an orthonormal set. }
        \STATE{Define $\Delta_* := \infty I$, $S_{*} := \{0\}$, $\delta_1 := -[\alpha_1 \ \cdots \ \alpha_p]^T$}
        \WHILE{$ \|\mbox{Proj}_{S_{*}^\perp} (A+\Delta_*) \|_F > \mbox{tol}_1$}
        \STATE{$\delta_{0} := \infty \delta_1$, $S_0:=S_1:=\{0\}$, $i:=1$}
        % \WHILE{$  (g_\varepsilon(\delta_{i-1},S_{i-1}) - g_\varepsilon(\delta_i,S_i)) > \mbox{tol}_2$}
        \WHILE{$ \|\delta_i-\delta_{i-1}\| > \mbox{tol}_2$} %\WHILE{$ \|\delta_i-\delta_{i-1}\| > \mbox{tol}_2$} %Can prove termination of the whole thing, not just inner loop? What is the current conv result even saying?
        \STATE{Update $i := i+1$}
        \STATE{Compute $S_i \in \arg\min_{S \in \mathfrak{S}_n} g_\varepsilon(\delta_{i-1},S)$}
        \STATE{Compute $\delta_{{i}} := \arg\min_{\delta \in \F^p} g_\varepsilon(\delta,S_i)$}
        \ENDWHILE
        \STATE{Update $\Delta_* := \sum_{l=1}^p B^{(l)} (\delta_i)_l $}
        \STATE{Update $\delta_1 := \delta_i$}
        \STATE{Update $S_* := S_{i}$}
        \STATE{Update $\varepsilon := k\varepsilon$}
        \ENDWHILE
        \RETURN $\Delta_*$
    \end{algorithmic}
\end{algorithm}

\begin{proposition}\label{prop:converge_reg}
    Algorithm~\ref{alg:reg} terminates after finitely many iterations for every $\mbox{tol}_1,\mbox{tol}_2>0$ and $k\in(0,1)$. If the inner iteration is continued indefinitely for a fixed $\varepsilon$, the sequence $(\delta_i,S_i)$
    has at least one accumulation point. Every accumulation
    point globally minimizes $g_\varepsilon$ with respect
    to each argument separately.
\end{proposition}

\begin{proof}
    For a fixed $\varepsilon$, the sequence $(g_\varepsilon(\delta_i,S_i))_{i\geq 2}$ is non-increasing and bounded below by zero, and hence converges if continued indefinitely. 
    
Expanding the two squared norms in $g_\varepsilon(\delta_i,S_{i+1})$ by decomposing  $\delta_i=\delta_{i+1}+(\delta_i-\delta_{i+1})$, the linear terms cancel by the optimality condition for $\delta_{i+1}$. Hence,
        \[
        g_\varepsilon(\delta_i,S_{i+1})
        =
        g_\varepsilon(\delta_{i+1},S_{i+1}) +
        \|\delta_i-\delta_{i+1}\|^2
        +\frac{1}{\varepsilon}
        \|M_{S_{i+1}}(\delta_i-\delta_{i+1})\|^2.
    \]
    This implies that 
    \[
        g_\varepsilon(\delta_i,S_{i+1}) - g_\varepsilon(\delta_{i+1},S_{i+1})
        \geq
        \|\delta_i-\delta_{i+1}\|^2.
    \]
    Hence,
    \[
        \|\delta_i-\delta_{i+1}\|^2
        \leq g_\varepsilon(\delta_i,S_i)
        -g_\varepsilon(\delta_{i+1},S_{i+1})
        \longrightarrow 0,
    \]
    where we used $g_\varepsilon(\delta_i,S_{i+1}) \leq g_\varepsilon(\delta_i,S_i)$. Therefore, $\|\delta_i-\delta_{i+1}\|\to0$, which
    proves finite termination of the inner iteration
    for every $\mbox{tol}_2>0$.

    Since $\Delta=-A$ is feasible, each minimization over
    $\delta$ gives
    \[
        \|\operatorname{Proj}_{S_i^\perp}(A+\Delta_i)\|_F^2
        \leq \varepsilon g_\varepsilon(\delta_i,S_i)
        \leq \varepsilon\|A\|_F^2.
    \]
    In the outer iteration, $\varepsilon \rightarrow 0$. Hence, after finitely many outer iterations, the residual tolerance $\mbox{tol}_1$ gets triggered, proving termination.

    For the accumulation-point statement, consider again
    the inner iteration with fixed $\varepsilon$. We have
    \[
        \|\delta_i\|^2
        \leq g_\varepsilon(\delta_i,S_i)
        \leq\|A\|_F^2.
    \]
    Together with the compactness of $\mathfrak{S}_n$,
    this guarantees the existence of an accumulation
    point. Let $(\delta_*,S_*)$ be any accumulation point, and choose
a subsequence such that $(\delta_{i_j},S_{i_j})\to(\delta_*,S_*)$. Since $\|\delta_i-\delta_{i-1}\|\to0$, we also have
    $\delta_{i_j-1}\to\delta_*$. For arbitrary
    $\delta\in\F^p$ and $S\in\mathfrak{S}_n$, the
    minimization steps give $ g_\varepsilon(\delta_i,S_i) \leq g_\varepsilon(\delta,S_i)$ and $g_\varepsilon(\delta_{i-1},S_i) \leq g_\varepsilon(\delta_{i-1},S)$. Taking $i=i_j$ and letting $j \rightarrow \infty$ in these inequalities, joint continuity shows that $(\delta_*,S_*)$ globally minimizes $g_\varepsilon$ with respect to each argument separately.
\end{proof}

\subsection{Augmented Lagrangian} \label{sec:lagrangian}
% \textcolor{red}{Subtle point: Does augmented lagrangian even work for non-Riemannian feasible sets? I do not think so. In the interest of space, I do not make further additions to the below}

In order to attain the exact solution in the Tikhonov regularization approach, it is necessary to let $\varepsilon \to 0$. The benefit of a method based on the augmented Lagrangian formulation is that this is not required; however, this requires the inclusion of an additional dual variable $y$ to the objective function that needs to be updated appropriately (see \cite[Section~4.2.2]{Bertsekas} for details).

The augmented Lagrangian formulation for the objective function $g_\varepsilon$ is defined as
\begin{align} \label{eq:lagrangian}
    g_{\varepsilon,y}(\delta,S) = \| \delta \|^2 + \frac{1}{\varepsilon} \| M_S \delta - r_S \|^2 + 2 \mbox{Re}\langle y, M_S \delta - r_S \rangle.
\end{align}
Note that it is still possible to find a closed form expression for the minimizing $\delta$, when $\varepsilon, y$ and $S$ are fixed. This can be seen by expressing $g_{\varepsilon,y}(\delta,S)$ as
\begin{align} \label{eq:lagrangian_simplified}
    g_{\varepsilon,y}(\delta,S) = \| \delta \|^2 + \frac{1}{\varepsilon} \| M_S \delta - r_S + \varepsilon y \|^2 - \varepsilon \|y \|^2.
\end{align}
The last term does not affect the minimizing $\delta$ and $S$ and can hence be ignored. As such, the minimizer with respect to $\delta$ can be computed as before with \eqref{deltastar} by simply updating the right hand side $r_S$ to include the term $-\epsilon y$.

In contrast to the Tikhonov regularized case, finding a minimizer $S_* \in \mathfrak{S}_n$ for \eqref{eq:lagrangian} does not have an easy-to-compute closed-form expression, when $\delta$ is kept fixed. For this reason, we resort to finding the minimizing value for $S$ numerically. We do this by restricting $S$ to the set $\{\mathcal{S}_v: \|v\|=1\}$, and numerically minimizing the function $(\delta,v) \mapsto g_{\varepsilon,y}(\delta,\mathcal{S}_v)$ under the constraint $\| v\| = 1$. For this, we use a Riemannian trust-regions method \cite{AbsilBaker} implemented in Manopt \cite{BoumalMishraAbsil}. Note that, in this case, $\|\mbox{Proj}_{\mathcal{S}_{v}^\perp} (A+\Delta) \|_F = \| (A+\Delta)v \|$.

% \begin{remark}
%     Restricting to the set $\{\mathcal{S}_v: \|v\|=1\}$ is beneficial for numerical optimization since its elements can be parametrized by the unit sphere $\{v \in \F^n: \|v\|=1\}$, which is a Riemannian manifold. Moreover, we will see in Section \ref{sec:convergence} that this also helps us say more about the convergence properties of the resulting algorithm.
% \end{remark}

An algorithm based on the augmented Lagrangian approach is as follows. For fixed $\varepsilon$ and $y$, we minimize the function $g_{\varepsilon,y}(\delta,\mathcal{S}_v)$ with respect to $\delta, v$. After the stopping criterion for the inner iteration is satisfied, we use the resulting point $(\delta_*,\mathcal{S}_{v_*})$ to set $y := y + \frac{1}{\varepsilon}(M_{\mathcal{S}_{v_*}} \delta_* - r_{\mathcal{S}_{v_*}}).$ We then perform the usual update on the parameter $\varepsilon$. We repeat this process until the norm of the constraint $\|(A+ \Delta) v\|$ reaches some prescribed tolerance. This procedure is outlined in Algorithm \ref{alg:lagrangian}. %Note that Proposition \ref{prop:converge_reg} can be extended to this algorithm as well. 
%For fixed $\varepsilon$ and $y$, the same argument as in Proposition~\ref{prop:converge_reg} proves finite termination of the inner iteration since the additional multiplier term is linear in $\delta$ and the objective function remains bounded from below.

\begin{algorithm}
    \caption{Augmented Lagrangian approach}
    \label{alg:lagrangian}
    \begin{algorithmic}
        \STATE{\textbf{Input:} initial value for $\varepsilon$, tolerances $\mbox{tol}_1, \mbox{tol}_2,$ \\
        \quad a decrease parameter $k$, and a matrix $A = \sum_{l=1}^p B^{(l)} (\alpha)_l,$ \\
        \quad where $\{B^{(l)}\}_{l=1}^p$ is an orthonormal set.}
        \STATE{Define $\Delta_* := \infty I$, $S_{*} := \{0\}$, $y:=0$, $\delta_1 := -[\alpha_1 \ \cdots \ \alpha_p]^T$ }
        \WHILE{$ \|\mbox{Proj}_{S_{*}^\perp} (A+\Delta_*) \|_F > \mbox{tol}_1$}
        \STATE{Define $\delta_{0} := \infty \delta_1$, $v_0 := v_1 := e_1$, $i:=1$}
        \WHILE{$\|\delta_i-\delta_{i-1}\| > \mbox{tol}_2$}
        \STATE{Update $i := i+1$}
        \STATE{Compute $v_i \in \arg\min_{\|v\| = 1} g_{\varepsilon,y}(\delta_{i-1},\mathcal{S}_{v})$}
        \STATE{Compute $\delta_{{i}} := \arg\min_{\delta \in \F^p} g_{\varepsilon,y}(\delta,\mathcal{S}_{v_{i}})$}
        \ENDWHILE
        \STATE{Update $\Delta_* := \sum_{l=1}^p B^{(l)} (\delta_i)_l $}
        \STATE{Update $\delta_1 := \delta_i$}
        \STATE{Update $S_* := \mathcal{S}_{v_i}$}
        \STATE{Update $y := y + \frac{1}{\varepsilon}(M_{\mathcal{S}_{v_{i}}} \delta_i - r_{\mathcal{S}_{v_{i}}})$}
        \STATE{Update $\varepsilon := k\varepsilon$}
        \ENDWHILE
        \RETURN $\Delta_*$
    \end{algorithmic}
\end{algorithm}

\begin{proposition}\label{prop:convergence_inner}
    Assume that each numerical minimization with respect to $v$ in the inner iteration of Algorithm~\ref{alg:lagrangian} reaches a stationary point. For fixed $\varepsilon$ and $y$, the inner iteration terminates after finitely many iterations for every $\mbox{tol}_2>0$. If the inner iteration is continued indefinitely, the sequence
    $(\delta_i,v_i)$ has at least one accumulation point.
    Every accumulation point globally minimizes
    $g_{\varepsilon,y}(\delta,\mathcal{S}_v)$ with respect to
    $\delta$ and is stationary with respect to $v$ on the unit sphere.
\end{proposition}
\begin{proof}
The proof can be adapted from that of Proposition~\ref{prop:converge_reg} with
$g_{\varepsilon,y}(\delta,\mathcal{S}_v)$ in place of $g_\varepsilon(\delta,S)$.
\end{proof}

\subsection{Convergence} \label{sec:convergence}
In this section, we give conditions under which accumulation points of the outer iteration of Algorithms~\ref{alg:reg} and \ref{alg:lagrangian} satisfy the KKT conditions of the original problem \eqref{eq:problem}. These results are specific to the setting where the singular spaces $S$ are restricted to the set $\{\mathcal{S}_v: \|v\|=1\}$.

There exist convergence results for augmented Lagrangian methods on arbitrary Riemannian manifolds \cite{LiuBoumal}, of which the unit sphere is one instance\footnote{More accurately, we need that the Cartesian product of $\mathcal{T} \subset \F^{n\times n}$ with the unit sphere is a Riemannian manifold, which is also true.}. %In \cite{LiuBoumal}, the authors have outlined convergence results for the augmented Lagrangian method on Riemannian manifolds. 
A similar application of convergence theory for Riemannian augmented Lagrangian methods to the structured distance to singularity problem appears in \cite{oracle}.
We state the result \cite[Proposition 3.2]{LiuBoumal} below as Proposition \ref{prop:convergence}. The result refers to Linear Independence Constraint Qualifications (LICQ) and First-Order Necessary Conditions (KKT conditions), whose technical definitions can be found in \cite[Equations (4.3) and (4.8)]{Zhang}.

\begin{proposition}[\cite{LiuBoumal}] \label{prop:convergence}
    Let $\Omega$ denote the set of feasible points of the problem \eqref{eq:problem}. Consider the Riemannian augmented Lagrangian method \cite[Algorithm 1]{LiuBoumal} with a sequence of tolerances $(\tau_k)_{k \in \Z^+}$ such that  $ \mbox{lim}_{k \rightarrow \infty} \tau_k = 0$. If at each iteration $k$ the subsolver produces a point $x_{k+1}$ satisfying
    \begin{equation}
        \label{eq:stop_gradient}
        \left\|\operatorname{grad}_x \mathcal{L}_{\varepsilon_k}\left(x_{k+1}, y^k\right)\right\| \leq \tau_k,
    \end{equation}
    and if the sequence $\left\{x_k\right\}_{k=0}^{\infty}$ has a limit point $\overline{x} \in \Omega$ where LICQ conditions are satisfied, then $\overline{x}$ satisfies KKT conditions of the original constrained minimization problem.
\end{proposition}

We apply Proposition~\ref{prop:convergence} with $x = (\delta, v)$ and $\mathcal{L}_{\varepsilon}\left(x, y\right) = g_{\varepsilon,y}(\delta,\mathcal{S}_v)$.
% Let us set $x = (\delta, \mathcal{S}_v)$ and $\mathcal{L}_{\varepsilon}\left(x, y\right) = g_{\varepsilon,y}(\delta,\mathcal{S}_v)$ in Proposition \ref{prop:convergence}. 
%Proposition~\ref{prop:convergence_inner} shows that every accumulation point of the inner iteration of Algorithm~\ref{alg:lagrangian} is stationary and therefore
% satisfies condition~\eqref{eq:stop_gradient} in Proposition~\ref{prop:convergence}.
%Proposition \ref{prop:convergence_inner} shows that the condition \eqref{eq:stop_gradient} in Proposition \ref{prop:convergence} is satisfied at the limit point whenever the inner iteration of Algorithm \ref{alg:lagrangian} converges to a point. %In practice, this requires a suitable choice of tolerances within the algorithm. 
If the inner iterations are terminated so that \eqref{eq:stop_gradient} holds with $\tau_k\to0$, then every feasible
%Assuming that each inner iteration converges, this implies that a 
limit point of the outer iteration in Algorithm \ref{alg:lagrangian} satisfies the KKT conditions of the original problem \eqref{eq:problem} when the additional LICQ conditions \cite[Equations (4.3) and (4.8)]{Zhang} are satisfied. The same convergence result applies for Algorithm \ref{alg:reg} when restricted to sets of type $\mathcal{S}_v$, in which case $y=0$ in each iteration (see \cite[Algorithm 1]{LiuBoumal}).

The LICQ conditions are derived in \cite{oracle} for this problem. These conditions are equivalent with assuming that the matrix
\begin{equation} \label{eq:LICQ}
    \begin{bmatrix}
        M_Y^* \\
        (A + \Delta)^*
    \end{bmatrix}
\end{equation}
has full rank at the feasible limit point. Here, $M_Y$ is defined as in Corollary \ref{cor:fastest_eval}.

\section{Numerical experiments} \label{sec:numerical}

Numerical experiments were performed on MATLAB R2025b, by using Manopt 8.0 \cite{BoumalMishraAbsil}, and on a machine equipped with an Intel Core i5-9400F processor and 16GB of RAM. The source code for the method presented in this paper can be found in the GitHub repository \href{https://github.com/NymanLauri/structured-distance-to-singularity}{github.com/NymanLauri/structured-distance-to-singularity}. The code is under the MIT license. In the numerical experiments of this paper, we use the version \texttt{v1.0.0 - Initial Release} of the code. Unless stated otherwise, the experiments were run by using the default values of the parameters. For an efficient computation of the smallest singular triplet of a matrix, the function \texttt{svdmin.m}, written by Ethan N. Epperly, Yuji Nakatsukasa and Taejun Park, implements a solver based on the work of \cite{yujisvd}. This implementation is used in the Tikhonov regularized approach in Algorithm~\ref{alg:reg}, when using singular spaces of the form $\mathcal{S}_v$ or $\mathcal{S}_v^*$. The augmented Lagrangian approach of Algorithm~\ref{alg:lagrangian} does not use
\texttt{svdmin.m}, since the update of $v$ is computed using a Riemannian trust-regions method. This approach is used in most of the comparisons against other methods below.

\subsection{Comparison of singular spaces} \label{sec:compare_sing}
First, we compare four strategies for how to choose the singular spaces in Procedure \ref{proc:construct_reg} of Algorithm \ref{alg:reg}: (i) construct a space in the set $\{\mathcal{S}_v: \|v\|=1\}$ in each iteration (ii) construct a space in the set $\{\mathcal{S}_v^*: \|v\|=1\}$ in each iteration (iii) alternate between spaces in the sets $\{\mathcal{S}_v: \|v\|=1\}$ and $\{\mathcal{S}_v^*: \|v\|=1\}$ in successive iterations (iv) alternate between spaces in the set $\{\mathcal{S}_v: \|v\|=1\}$ and singular subspaces that have underlying zero submatrix of size $(n-1) \times 2$. We generate randomly $5 \cdot 10^3$ Toeplitz matrices of size $n=100$ such that each diagonal is drawn independently from the unit normal distribution, and we compute the Toeplitz structured distance to singularity for these. The results are shown in Table \ref{table:compare_alternate}. The computed distances show only marginal differences, while strategies (iii) and (iv) required a slightly larger number of iterations. The running time of (iv) is significantly higher since the computation of the matrix $M$ in Proposition \ref{prop:fast_eval} is done via a full \texttt{svd}, in contrast to the one in Corollary \ref{cor:fastest_eval}, which only requires the smallest singular triplet. The running time of (iii) is also approximately 10 \% larger than that of (i) and (ii), despite the fact that (iii) essentially alternates between strategies (i) and (ii). %The fact that (iii) performs worse than (i) and (ii) is surprising, given that (iii) essentially alternates between strategies (i) and (ii). 
%These results suggest using either strategy (i) or (ii) by default. 
Based on these results, we will opt for strategy (i) in the following experiments.

\begin{table}[h!]
    \caption{Comparison of strategies (i)-(iv) to find the Toeplitz structured distance to singularity for a sample of $5 \cdot 10^3$ randomly generated Toeplitz structured matrices of size $n=100$. The table shows, for each strategy, the median values for the computed distances, total number of inner iterations, as well as running times in seconds.}
    %The table also records the corresponding running times.}
    \label{table:compare_alternate}
    \begin{center}
        \begin{tabular}{|c||c|c|c|}
            \hline
            Strategy & Distance & \# iterations & Running time (s) \\
            \hline
            (i)      & 0.5812   & 629           & 0.4181           \\
            \hline
            (ii)     & 0.5812   & 629           & 0.4223           \\
            \hline
            (iii)    & 0.5815   & 655           & 0.4651           \\
            \hline
            (iv)     & 0.5815   & 660           & 2.1310           \\
            \hline
            % (42,28) &  &  \\
            % \hline
        \end{tabular}
    \end{center}
\end{table}

\subsection{Sensitivity to initialization}
The choice of the initial point can significantly affect the value to which the algorithm converges. We demonstrate this numerically for Algorithm \ref{alg:reg} that uses strategy (i) of Section~\ref{sec:compare_sing}. We do this by generating randomly $10$ Toeplitz matrices of size $100 \times 100$ such that each diagonal is drawn independently from the unit normal distribution, and applying the algorithm $20$ times to each matrix, each time starting from a different initialization vector $v$. These unit vectors $v$ are drawn independently from the uniform distribution.  

Figure~\ref{fig:initialization_sensitivity} shows the results. For six of the matrices, the algorithm converges to the same distance for each starting point. For the remaining four matrices, the algorithm converges to two different values that are sometimes significantly different from each other. While the total amount of innermost iterations shows slightly more variation than the computed distances, it follows the same general trend. %This highlights that the algorithm can occasionally get trapped in suboptimal local minima if an unfavorable starting point is chosen.

\begin{figure}[htpb]
    \centering
    \begin{tikzpicture}
        \begin{axis}[
            width=0.48\textwidth,
            height=6cm,
            xlabel={Matrix index},
            ylabel={Computed distance},
            xtick={1,2,3,4,5,6,7,8,9,10},
            grid=major
        ]
            \addplot [
                only marks,
                mark=*,
                mark options={fill=blue!70!white, draw=blue!50!black, opacity=0.5},
                mark size=1.5pt
            ] table [x=x_jit, y=dist] {initialisation_sensitivity.dat};
        \end{axis}
    \end{tikzpicture}%
    \hfill
    \begin{tikzpicture}
        \begin{axis}[
            width=0.48\textwidth,
            height=6cm,
            xlabel={Matrix index},
            ylabel={\# iterations},
            xtick={1,2,3,4,5,6,7,8,9,10},
            grid=major
        ]
            \addplot [
                only marks,
                mark=square*,
                mark options={fill=orange!80!white, draw=orange!60!black, opacity=0.5},
                mark size=1.5pt
            ] table [x=x_jit, y=iters] {initialisation_sensitivity.dat};
        \end{axis}
    \end{tikzpicture}
    \caption{Distribution of the computed distance to singularity (left) and the number of inner iterations (right) for 20 different initializations. The results are shown for a sample of 10 different randomly generated Toeplitz structured matrices of size $100 \times 100$. %A horizontal jitter is added to the discrete matrix indices to make overlapping points visible. Matrices for which the points spread vertically exhibit sensitivity to the initial value.
    }
    \label{fig:initialization_sensitivity}
\end{figure}

\subsection{Effect of inner iteration tolerance}
We next examine the effect of the stopping tolerance $\mbox{tol}_2$ of the inner iteration of Algorithm \ref{alg:reg}. %denoted by $\mbox{tol}_2$ in the pseudocode.
% An important parameter of Algorithm \ref{alg:reg} is the stopping tolerance of the innermost iteration, denoted by $\mbox{tol}_2$ in the pseudocode. In this numerical experiment, we investigate how this tolerance affects the algorithm. 
To do this, we run the algorithm on 10 Toeplitz matrices (generated similarly as before) of size $100 \times 100$ for a range of tolerances from $10^{-2}$ to $10^{-12}$. 

Figure~\ref{fig:tolerance_sensitivity} shows the median relative error in the computed distance (computed by treating the result for $10^{-12}$ as the exact solution) and the median total number of inner iterations required. The default tolerance of $10^{-8}$ used in other experiments achieves a median relative accuracy close to $10^{-6}$.  %This corresponds to having at least the first four digits correct, which is a decent accuracy in practice.\footnote{\color{red}In many cases, one is interested in the degree to which the given input is singular. An output within an order of magnitude of the true distance might even be satisfactory, and a relative error within one part per thousand would not in most cases be significant.} 
Decreasing the tolerance further would offer improvement in accuracy, while also increasing the computational cost. The error in the distance for tolerance $10^{-12}$ is omitted, as this would misleadingly yield the value zero (given that error is measured relative to itself).

\begin{figure}[htpb]
    \centering
    \begin{tikzpicture}
        \begin{axis}[
            xmode=log,
            ymode=log,
            width=0.47\textwidth,
            height=6cm,
            xlabel={Inner iteration tolerance},
            ylabel={Relative error},
            x dir=reverse,
            grid=major
        ]
            % Medians from tolerance_sensitivity.dat.
            \addplot [mark=*, mark size=1.5pt, blue!70!white, draw=blue!50!black, thick] table [x=tol, y=err_median] {tolerance_sensitivity.dat};
        \end{axis}
    \end{tikzpicture}%
    \hfill
    \begin{tikzpicture}
        \begin{axis}[
            xmode=log,
            width=0.47\textwidth,
            height=6cm,
            xlabel={Inner iteration tolerance},
            ylabel={\# iterations},
            x dir=reverse,
            grid=major
        ]
            % Medians from tolerance_sensitivity.dat.
            \addplot [mark=square*, mark size=1.5pt, orange!80!white, draw=orange!60!black, thick] table [x=tol, y=iter_median] {tolerance_sensitivity.dat};
        \end{axis}
    \end{tikzpicture}
    \caption{Effect of the stopping tolerance for the innermost loop of Algorithm \ref{alg:reg}. The figure plots the median relative error in the computed distance to singularity (left) and the total number of inner iterations taken (right) for a sample of 10 randomly generated Toeplitz structured matrices of size $100 \times 100$. %The reference distances for the relative error were computed at a very tight tolerance of $10^{-12}$.
    The values for the relative error were computed by treating the result for $10^{-12}$ as the exact solution.
    }
    \label{fig:tolerance_sensitivity}
\end{figure}

\subsection{Effect of the regularization update factor}
In Algorithm \ref{alg:reg}, the regularization parameter $\varepsilon$ is scaled by the update factor $k$ in each outer iteration. In the practical implementation of the algorithm, we use a default update factor of $k = 10^{-0.5}$. In this experiment, we investigate the effect of this update factor $k$ on the algorithm. We test a range of factors from $10^{-2.0}$ to $10^{-0.1}$ for 10 Toeplitz matrices (generated similarly as before) of size $100 \times 100$.

Figure~\ref{fig:update_factor_sensitivity} shows the median relative error in the computed distance (treating the result for the slowest update factor $k = 10^{-0.1}$ as the exact solution) and the total number of inner iterations required. Smaller update factors reduce the number of iterations but produce less accurate output. %Conversely, a very tight factor like $0.79$ produces the most accurate results but requires a significantly larger number of inner iterations (a median of nearly $1400$). 
The default value $k = 10^{-0.5}$ achieves a median relative accuracy of approximately $10^{-5}$.

\begin{figure}[htpb]
    \centering
    \begin{tikzpicture}
        \begin{axis}[
            xmode=log,
            ymode=log,
            width=0.48\textwidth,
            height=6cm,
            xlabel={Update factor $k$},
            ylabel={Relative error},
            x dir=reverse,
            grid=major
        ]
            \addplot [mark=*, mark size=1.5pt, blue!70!white, draw=blue!50!black, thick] table [x=factor, y=err_median] {update_factor_sensitivity.dat};
        \end{axis}
    \end{tikzpicture}%
    \hfill
    \begin{tikzpicture}
        \begin{axis}[
            xmode=log,
            width=0.48\textwidth,
            height=6cm,
            xlabel={Update factor $k$},
            ylabel={\# iterations},
            x dir=reverse,
            grid=major
        ]
            \addplot [mark=square*, mark size=1.5pt, orange!80!white, draw=orange!60!black, thick] table [x=factor, y=iter_median] {update_factor_sensitivity.dat};
        \end{axis}
    \end{tikzpicture}
    \caption{Effect of the update factor $k$ for the regularization parameter $\varepsilon$ in Algorithm \ref{alg:reg}. The figure plots the median relative error in the computed distance (left) and the median total number of inner iterations (right) for a sample of 10 randomly generated Toeplitz structured matrices of size $100 \times 100$. The values for the relative error were computed by treating the result for an update factor of $k = 10^{-0.1}$ as the exact solution.}
    \label{fig:update_factor_sensitivity}
\end{figure}

\subsection{Comparison of regularization approaches}
We next compare the Ti\-khonov regularization approach of Algorithm \ref{alg:reg} and the augmented Lagrangian approach of Algorithm \ref{alg:lagrangian} for Toeplitz structured matrices of various sizes in the range $0 < n \leq 500$. Figure \ref{fig:lagrangian_comparison} visualizes the median of 40 runs for the running times as well as total inner iteration counts. The computed distances for both algorithms were indistinguishable from each other, and were hence omitted from the picture. The running times, however, show an interesting trend: the Tikhonov regularization approach is significantly faster for matrices of size $n \leq 300$, while for larger sizes, the augmented Lagrangian approach becomes significantly faster. The total number of iterations is significantly smaller for the augmented Lagrangian approach throughout the whole interval for $n$. Interestingly, the total number of iterations increases slowly with $n$ for both approaches. %This suggests that both approaches scale very well for larger problems, while the augmented Lagrangian approach seems to scale better. %To support this conclusion, we next compare the method of this paper against existing, state-of-the-art methods.

\begin{figure}[h!]
    \begin{tikzpicture}
        \begin{axis}[
                legend pos = north west,
                width=0.5*\linewidth,
                xtick={0, 100, 200, 300, 400, 500},
                % log ticks with fixed point,
                % ytick={0.1, 0.3, 1, 3, 10},
                height= 6cm,
                title={{\small{Median of time (linear)}}},
                xlabel={size ($n$)},
                ylabel={time (s)}
            ]
            \addplot[red, densely dashdotdotted, thick] table[y index = 1]
                {lagrangian_times.dat};
            \addplot[blue, dashed, thick] table[y index = 2]
                {lagrangian_times.dat};
            \addplot[only marks, red, mark size =1.5pt,  mark=square*,
                mark options={fill=red, solid}] table[y index = 1]
                {lagrangian_times.dat};
            \addplot[only marks, blue, dashed, mark size =1.5pt,  mark=*, mark options={fill=blue, solid}] table[y index = 2]
                {lagrangian_times.dat};
            \legend{{\small{Tikhonov}}, {\small{Lagrangian}}}
        \end{axis}
    \end{tikzpicture}
    \begin{tikzpicture}
        \begin{semilogyaxis}[
                legend pos = south east,
                width=0.5*\linewidth,
                xtick={0, 100, 200, 300, 400, 500},
                % log ticks with fixed point,
                % ytick={0.1, 0.3, 1, 3, 10},
                height= 6cm,
                title={{\small{Median of time (logarithmic)}}},
                xlabel={size ($n$)},
                ylabel={time (s)}
            ]
            \addplot[red, densely dashdotdotted, thick] table[y index = 1]
                {lagrangian_times.dat};
            \addplot[blue, dashed, thick] table[y index = 2]
                {lagrangian_times.dat};
            \addplot[only marks, red, mark size =1.5pt,  mark=square*,
                mark options={fill=red, solid}] table[y index = 1]
                {lagrangian_times.dat};
            \addplot[only marks, blue, dashed, mark size =1.5pt,  mark=*, mark options={fill=blue, solid}] table[y index = 2]
                {lagrangian_times.dat};
            \legend{{\small{Tikhonov}}, {\small{Lagrangian}}}
        \end{semilogyaxis}
    \end{tikzpicture}
    \begin{tikzpicture}
        \begin{axis}[
                % legend pos = north west,
                legend style={at={(0.97,0.75)}},
                width=\linewidth,
                xtick={0, 50, 100, 150, 200, 250, 300, 350, 400, 450, 500},
                ytick={0, 200, 400, 600, 800},
                height=4cm,
                title={{\small{Median number of iterations}}},
                xlabel={size ($n$)},
                ylabel={\# iterations}
            ]
            \addplot[red, densely dashdotdotted, thick] table[y index = 1]
                {lagrangian_counts.dat};
            \addplot[blue, dashed, thick] table[y index = 2]
                {lagrangian_counts.dat};
            \addplot[only marks, red, mark size =1.5pt,  mark=square*,
                mark options={fill=red, solid}] table[y index = 1]
                {lagrangian_counts.dat};
            \addplot[only marks, blue, dashed, mark size =1.5pt,  mark=*, mark options={fill=blue, solid}] table[y index = 2]
                {lagrangian_counts.dat};
            \legend{{\small{Tikhonov}}, {\small{Lagrangian}}}
        \end{axis}
    \end{tikzpicture}
    \caption{Comparison between the Tikhonov regularization approach and the augmented Lagrangian approach for Toeplitz structured matrices of increasing sizes. The figure plots the median running time in seconds (above) and the median number of total inner iterations (below) for a sample of 40 randomly generated matrices for each size $n$.}
    \label{fig:lagrangian_comparison}
\end{figure}

%Next, we compare the method of this paper against existing methods. 
\subsection{Comparison against existing methods}
We compare the proposed method with the Riemann-Oracle method~\cite{oracle}, the matrix differential equation (ODE) approach~\cite{Sicilia}, and the nonlinear equations
approach~\cite{poloni2026}. These methods use formulations designed specifically for loss of full rank. Riemann-Oracle uses variable projection and optimises over a single null vector, while the methods of~\cite{Sicilia,poloni2026} exploit representations based on projecting rank-one matrices onto the prescribed structure.

For a square matrix of size $n$, the bound $\operatorname{rank}(B)\leq n-1$ can be imposed by requiring a single nonzero kernel vector. A factorisation $B=XY^*$ with the same rank bound uses factors $X,Y\in\mathbb{F}^{n\times(n-1)}$. Both factors are nearly as large as the original matrix, resulting in a large number of optimization variables. Factorisation approaches such as~\cite{IshUM14} are therefore better suited to small target ranks. Kernel-based methods for structured low-rank
approximation~\cite{MarU13missing,MarU14software,UseM14manifold,UseM14}
are better suited to small target nullities. In our numerical comparisons, we focus on the three recent methods introduced above, of which Riemann-Oracle is a kernel-based method.

While the framework of \cite{Sicilia} in theory works for any linear structure, their algorithm focuses on sparse structures. For this reason, we compare against the method of \cite{Sicilia} only in the sparse case. 
%The implementation of~\cite{Sicilia} considered here focuses on sparse structures, and we therefore include this comparison only in the sparse case. 
Comparisons with the more recent nonlinear equations approach~\cite{poloni2026} are presented separately in Section~\ref{sec:comp_newton}.

% Currently, the best two methods existing in the literature are arguably the Riemann-Oracle method outlined in \cite{oracle}, and the ODE-based approach outlined in \cite{Sicilia}. % \cite[Section V.8]{guglielmi_book}. 
% The Riemann-Oracle method is a flexible framework for solving various matrix nearness problems. The numerical experiments presented in \cite{oracle} make a strong case for it being the best algorithm currently in the literature. However, the Riemann-Oracle method has not been directly compared with the ODE-based approach of \cite{Sicilia} which utilizes an underlying rank-1 structure of the problem. %for the structured distance to singularity problem. %which avoids the computation of eigenvalues entirely, and hence has potential for being the most efficient algorithm currently. 
%While the framework of \cite{Sicilia} in theory works for any linear structure, their algorithm focuses on sparse structures. For this reason, we compare against the method of \cite{Sicilia} only in the sparse case.

For the method based on matrix differential equations of \cite{Sicilia}, the comparison was run using the MATLAB codes kindly provided by the authors. For the Riemann-Oracle method of \cite{oracle}, we use the code that is available at \href{https://github.com/fph/RiemannOracle}{github.com/fph/\linebreak RiemannOracle}. There exist two different formulations for the Riemann-Oracle meth\-od: one based on the penalty method, and one based on the augmented Lagrangian approach. In order for the comparison to be as fair as possible, we compare against both formulations.

The Riemann-Oracle method implements multiple different update strategies for their regularization parameter. In the numerical experiments of \cite{oracle}, the authors use an adaptive update strategy for the regularization parameter in the penalty method formulation, which is achieved by setting \texttt{options.epsilon\_decrease = ‘f’} in the \texttt{options} structure that is provided to the function. For the augmented Lagrangian method, they use the default decrease strategy. In our numerical comparisons, we use these same update strategies for the Riemann-Oracle algorithm.

In order for our comparisons to be as fair as possible, we make sure that the constraint violation\footnote{We measure constraint violation with the smallest singular value of the output.} in the numerical experiments is smaller for the method of this paper than for Riemann-Oracle or the ODE approach. To achieve this, we use a stopping criterion of $10^{-14}$ for the constraint in the Riemann-Oracle algorithm, which can be achieved by setting \texttt{options.stopping\_criterion = 1e-14}. When using this value for the stopping criterion for Riemann-Oracle, and the default stopping criterion for the ODE approach, the mean and median constraint violations in the numerical experiments were the smallest for the method of this paper.

\subsubsection{Toeplitz structures}
First, we perform a comparison for the Toeplitz structured distance to singularity problem. In this experiment, we compare the augmented Lagrangian approach of Algorithm \ref{alg:lagrangian} with the Riemann-Oracle method of \cite{oracle}. %The authors of \cite{oracle} explain that the augmented Lagrangian choice for the Riemann-Oracle method typically works the best for matrices of large size, and the documentation of their code also recommends its use. Our heuristic observations support this idea: we tried running this experiment for both, the penalty method approach as well as the augmented Lagrangian approach of \cite{oracle}, and the approach based on the penalty method formulation was significantly slower, which prevented us from performing the comparison with larger matrices. For these reasons, we compare our method against 
%In order for the comparison to be as fair as possible, we use the augmented Lagrangian formulation for the method of this paper as well, although Figure \ref{fig:lagrangian_comparison} suggests that the Tikhonov regularization approach would perform better for moderately sized matrices. 
To do this, we generate varying sizes of Toeplitz structured matrices by sampling the values of each diagonal independently from the unit normal distribution. For a sample of 40 matrices, Figure \ref{fig:RO_comparison} shows the median values for the computed distances to singularity as well as the running times. The method of this paper is substantially faster than the Riemann-Oracle algorithm, while the computed distances show only very minor differences. More precisely, the relative difference in the computed distances is less than $0.03$ for $n=10, 50$, and less than $10^{-6}$ for $n \geq 100$. The running time of the Riemann-Oracle algorithm became impractically slow for a statistical experiment at size $n=150$ for the penalty method formulation, and at size $n=200$ for the augmented Lagrangian formulation. At these sizes, the method of this paper offers an improvement of one to two orders of magnitude in running time. The logarithmic plot suggests that the difference in the running time becomes orders of magnitude wider for larger sizes.

% \begin{table}[h!]
% \caption{Comparison with the RO method \cite{oracle}. The table shows, for both algorithms, the computed distances to a nearest singular Toeplitz structured matrix for randomly generated matrices of sizes $n=100, 200, 300$. %The distances are measured relative to the norm of the input pencil. 
% The table also records the corresponding running times.}
% \label{table:structured_distance}
% \begin{center}
% \begin{tabular}{|c||c|c|c|c|}
% \hline
%  &\multicolumn{2}{c|}{Distance}&\multicolumn{2}{c|}{Running time (s)} \\
% \hline
% Size & This paper & RO & This paper & RO  \\
% \hline
% 100 &  0.0394 & 0.0394 &  0.35 & 24.47  \\
% \hline
% 200 &  0.7416 & 0.7415 &  2.24 & 81.68 \\
% \hline
% 300 &  0.6730 & 0.6693 &  6.96 & 795.75\\
% \hline
% \end{tabular}
% \end{center}
% \end{table}

\begin{figure}[h!]
    \begin{tikzpicture}
        \begin{axis}[
                % legend pos = north west,
                legend style={at={(0.97,0.7)}},
                width=0.5*\linewidth,
                xtick={0, 100, 200, 300, 400, 500},
                % log ticks with fixed point,
                % ytick={0.1, 0.3, 1, 3, 10},
                height= 6cm,
                title={{\small{Median of time (linear)}}},
                xlabel={size ($n$)},
                ylabel={time (s)}
            ]
            \addplot[red, densely dashdotdotted, thick, error bars/.cd, y dir=both, y explicit] table[y index = 1, y error minus index = 4, y error plus index = 5]
                {Toe_times_iqr.dat};
            \addplot[blue, dashed, thick, restrict x to domain=0:200, error bars/.cd, y dir=both, y explicit] table[y index = 2, y error minus index = 6, y error plus index = 7]
                {Toe_times_iqr.dat};
            \addplot[orange, densely dashdotdotted, thick, restrict x to domain=0:150, error bars/.cd, y dir=both, y explicit] table[y index = 3, y error minus index = 8, y error plus index = 9]
                {Toe_times_iqr.dat};
            \addplot[only marks, red, mark size =1.5pt,  mark=square*,
                mark options={fill=red, solid}] table[y index = 1]
                {Toe_times_iqr.dat};
            \addplot[only marks, blue, dashed, mark size =1.5pt,  mark=*, mark options={fill=blue, solid}, restrict x to domain=0:200] table[y index = 2]
                {Toe_times_iqr.dat};
            \addplot[only marks, orange, dashed, mark size =1.5pt,  mark=*, mark options={fill=orange, solid}, restrict x to domain=0:150] table[y index = 3]
                {Toe_times_iqr.dat};
            \legend{{\small{SSA}}, {\small{RO-L}}, {\small{RO-P}}}
        \end{axis}
    \end{tikzpicture}
    \begin{tikzpicture}
        \begin{semilogyaxis}[
                legend pos = south east,
                width=0.5*\linewidth,
                xtick={0, 100, 200, 300, 400, 500},
                % log ticks with fixed point,
                % ytick={0.1, 0.3, 1, 3, 10},
                height= 6cm,
                title={{\small{Median of time (logarithmic)}}},
                xlabel={size ($n$)},
                ylabel={time (s)}
            ]
            \addplot[red, densely dashdotdotted, thick, error bars/.cd, y dir=both, y explicit] table[y index = 1, y error minus index = 4, y error plus index = 5]
                {Toe_times_iqr.dat};
            \addplot[blue, dashed, thick, restrict x to domain=0:200, error bars/.cd, y dir=both, y explicit] table[y index = 2, y error minus index = 6, y error plus index = 7]
                {Toe_times_iqr.dat};
            \addplot[orange, densely dashdotdotted, thick, restrict x to domain=0:150, error bars/.cd, y dir=both, y explicit] table[y index = 3, y error minus index = 8, y error plus index = 9]
                {Toe_times_iqr.dat};
            \addplot[only marks, red, mark size =1.5pt,  mark=square*,
                mark options={fill=red, solid}] table[y index = 1]
                {Toe_times_iqr.dat};
            \addplot[only marks, blue, dashed, mark size =1.5pt,  mark=*, mark options={fill=blue, solid}, restrict x to domain=0:200] table[y index = 2]
                {Toe_times_iqr.dat};
            \addplot[only marks, orange, dashed, mark size =1.5pt,  mark=*, mark options={fill=orange, solid}, restrict x to domain=0:150] table[y index = 3]
                {Toe_times_iqr.dat};
            \legend{{\small{SSA}}, {\small{RO-L}}, {\small{RO-P}}}
        \end{semilogyaxis}
    \end{tikzpicture}
    \begin{tikzpicture}
        \begin{axis}[
                legend pos = south west,
                legend columns=0,
                width=\linewidth,
                xtick={0, 50, 100, 150, 200, 250, 300, 350, 400, 450, 500},
                ytick={0, 0.5, 1},
                ymin=-0.3, ymax=1.4,
                height=4cm,
                title={{\small{Median of computed distance}}},
                xlabel={size ($n$)},
                ylabel={distance}
            ]
            \addplot[red, densely dashdotdotted, thick, error bars/.cd, y dir=both, y explicit] table[y index = 1, y error minus index = 4, y error plus index = 5]
                {Toe_norms_iqr.dat};
            \addplot[blue, dashed, thick, restrict x to domain=0:200, error bars/.cd, y dir=both, y explicit] table[y index = 2, y error minus index = 6, y error plus index = 7]
                {Toe_norms_iqr.dat};
            \addplot[orange, densely dashdotdotted, thick, restrict x to domain=0:150, error bars/.cd, y dir=both, y explicit] table[y index = 3, y error minus index = 8, y error plus index = 9]
                {Toe_norms_iqr.dat};
            \addplot[only marks, red, mark size =1.5pt,  mark=square*,
                mark options={fill=red, solid}] table[y index = 1]
                {Toe_norms_iqr.dat};
            \addplot[only marks, blue, dashed, mark size =1.5pt,  mark=*, mark options={fill=blue, solid}, restrict x to domain=0:200] table[y index = 2]
                {Toe_norms_iqr.dat};
            \addplot[only marks, orange, dashed, mark size =1.5pt,  mark=*, mark options={fill=orange, solid}, restrict x to domain=0:150] table[y index = 3]
                {Toe_norms_iqr.dat};
            \legend{{\small{SSA}}, {\small{RO-L}}, {\small{RO-P}}}
        \end{axis}
    \end{tikzpicture}
    \caption{Comparison between the singular space approach of this paper (denoted by SSA) and the Riemann-Oracle method of \cite{oracle} (denoted by RO-L and RO-P) for Toeplitz structured matrices of increasing sizes. RO-L denotes the augmented Lagrangian formulation of \cite{oracle}, while RO-P denotes its penalty method formulation. The figure plots the median running time in seconds (above) and the median computed distance to singularity (below) for a sample of 40 randomly generated matrices for each size $n$. The error bars represent the interquartile range corresponding to the 25th and 75th percentiles. For distances, the interquartile ranges are almost identical for all methods.}
    \label{fig:RO_comparison}
\end{figure}

\subsubsection{Sparse structures}

Next, we compare against the Riemann-Oracle \linebreak meth\-od \cite{oracle} and the ODE approach \cite{Sicilia} for the sparsely structured distance to singularity problem. For sparse structures, we exploit sparsity to optimize the evaluation of the solution $\delta_*$ in Corollary \ref{cor:fastest_eval}, similarly to \cite[Section 5]{oracle}. However, the use of augmented Lagrangian requires an additional numerical optimization step, which is difficult to optimize and becomes the bottleneck. For this reason, we choose to use the Tikhonov regularization approach of Algorithm \ref{alg:reg} for this experiment. %The Riemann-Oracle algorithm does not have a similar problem, however, as the use of augmented Lagrangian does not introduce an additional numerical optimization step in their algorithm. In fact, heuristic observations for sparse tridiagonal structures showed that the Tikhonov regularized algorithm of the Riemann-Oracle takes significantly longer than the augmented Lagrangian algorithm, to the point that it was not practical to perform experiments with sizes larger than $60 \times 60$. For this reason, we perform the comparison by using the augmented Lagrangian formulation of Riemann-Oracle. However, even with the augmented Lagrangian approach, in some rare occasions the Riemann-Oracle method takes orders of magnitudes longer to finish than the median running time. These outliers do not significantly change the median values for the running times, but they make it difficult to perform a statistical comparison. To mitigate this issue, we set \texttt{options.maxiter = 100} for the Riemann-Oracle method, which limits the total number of iterations. For most input matrices of this numerical experiment, this stopping criterion does not get triggered.  

In this experiment, we randomly generate sparse matrices as follows. First, each element has probability $p$ to be non-zero, independently of each other. Then, each non-zero element is drawn independently from the unit normal distribution. %A combination of a small size as well as a low number of non-zero elements sometimes resulted in an error in the ODE method in our experiments\footnote{This error should be fixable, but we preferred not to make changes to the competing method. It is caused by the way in which the method constructs an initial point: if the matrix is very sparse, the method is more likely to construct the zero matrix as a starting point, which results in an error.}. In this experiment, we set $p=0.4$ and sample matrices of size $60 \leq n \leq 260$, in which case we observed no issues. 
In this experiment, we set $p=0.4$ and sample matrices of size $60 \leq n \leq 260$. For a sample of 40 matrices for each $n$, Figure \ref{fig:ODE_RO_comparison} shows the median values for the computed distances to singularity as well as the running times. Based on Figure \ref{fig:ODE_RO_comparison}, the method of this paper gives equally good output as the other methods, while decreasing the running time by one to two orders of magnitude.

\begin{figure}[h!]
    \begin{tikzpicture}
        \begin{axis}[
                legend pos = north west,
                width=0.5*\linewidth,
                xtick={0, 60, 100, 140, 180, 220, 260},
                % log ticks with fixed point,
                % ytick={0.1, 0.3, 1, 3, 10},
                height= 6cm,
                title={{\small{Median of time (linear)}}},
                xlabel={size ($n$)},
                ylabel={time (s)}
            ]
            \addplot[red, densely dashdotdotted, thick, error bars/.cd, y dir=both, y explicit] table[y index = 1, y error minus index = 5, y error plus index = 6]
                {Sparse_times_iqr.dat};
            \addplot[blue, dashed, thick, error bars/.cd, y dir=both, y explicit] table[y index = 2, y error minus index = 7, y error plus index = 8]
                {Sparse_times_iqr.dat};
            \addplot[orange, densely dashdotdotted, thick, error bars/.cd, y dir=both, y explicit] table[y index = 3, y error minus index = 9, y error plus index = 10]
                {Sparse_times_iqr.dat};
            \addplot[black!30!green, densely dashdotdotted, thick, error bars/.cd, y dir=both, y explicit] table[y index = 4, y error minus index = 11, y error plus index = 12]
                {Sparse_times_iqr.dat};
            \addplot[only marks, red, mark size =1.5pt,  mark=square*,
                mark options={fill=red, solid}] table[y index = 1]
                {Sparse_times_iqr.dat};
            \addplot[only marks, blue, dashed, mark size =1.5pt,  mark=*, mark options={fill=blue, solid}] table[y index = 2]
                {Sparse_times_iqr.dat};
            \addplot[only marks, orange, dashed, mark size =1.5pt,  mark=*, mark options={fill=orange, solid}] table[y index = 3]
                {Sparse_times_iqr.dat};
            \addplot[only marks, black!30!green, dashed, mark size =1.5pt,  mark=*, mark options={fill=black!30!green, solid}] table[y index = 4]
                {Sparse_times_iqr.dat};
            \legend{{\small{SSA}}, {\small{RO-L}}, {\small{RO-P}}, {\small{ODE}}}
        \end{axis}
    \end{tikzpicture}
    \begin{tikzpicture}
        \begin{semilogyaxis}[
                legend pos = north west,
                legend columns=2,
                % legend pos= outer north east,
                % legend style={at={(0.97,0.4)}},
                width=0.5*\linewidth,
                xtick={0, 60, 100, 140, 180, 220, 260},
                % log ticks with fixed point,
                % ytick={0.1, 0.3, 1, 3, 10},
                ymax=200,
                height= 6cm,
                title={{\small{Median of time (logarithmic)}}},
                xlabel={size ($n$)},
                ylabel={time (s)}
            ]
            \addplot[red, densely dashdotdotted, thick, error bars/.cd, y dir=both, y explicit] table[y index = 1, y error minus index = 5, y error plus index = 6]
                {Sparse_times_iqr.dat};
            \addplot[blue, dashed, thick, error bars/.cd, y dir=both, y explicit] table[y index = 2, y error minus index = 7, y error plus index = 8]
                {Sparse_times_iqr.dat};
            \addplot[orange, densely dashdotdotted, thick, error bars/.cd, y dir=both, y explicit] table[y index = 3, y error minus index = 9, y error plus index = 10]
                {Sparse_times_iqr.dat};
            \addplot[black!30!green, densely dashdotdotted, thick, error bars/.cd, y dir=both, y explicit] table[y index = 4, y error minus index = 11, y error plus index = 12]
                {Sparse_times_iqr.dat};
            \addplot[only marks, red, mark size =1.5pt,  mark=square*,
                mark options={fill=red, solid}] table[y index = 1]
                {Sparse_times_iqr.dat};
            \addplot[only marks, blue, dashed, mark size =1.5pt,  mark=*, mark options={fill=blue, solid}] table[y index = 2]
                {Sparse_times_iqr.dat};
            \addplot[only marks, orange, dashed, mark size =1.5pt,  mark=*, mark options={fill=orange, solid}] table[y index = 3]
                {Sparse_times_iqr.dat};
            \addplot[only marks, black!30!green, dashed, mark size =1.5pt,  mark=*, mark options={fill=black!30!green, solid}] table[y index = 4]
                {Sparse_times_iqr.dat};
            \legend{{\small{SSA}}, {\small{RO-L}}, {\small{RO-P}}, {\small{ODE}}}
        \end{semilogyaxis}
    \end{tikzpicture}
    \begin{tikzpicture}
        \begin{axis}[
                legend pos = north east,
                legend columns=0,
                width=\linewidth,
                yticklabel style={
                        /pgf/number format/fixed,
                        /pgf/number format/precision=3
                    },
                xtick={0, 60, 100, 140, 180, 220, 260},
                ytick={0, 0.05, 0.1},
                ymin=-0.01, ymax=0.13,
                height=4cm,
                title={{\small{Median of computed distance}}},
                xlabel={size ($n$)},
                ylabel={distance}
            ]
            \addplot[red, densely dashdotdotted, thick, error bars/.cd, y dir=both, y explicit] table[y index = 1, y error minus index = 5, y error plus index = 6]
                {Sparse_norms_iqr.dat};
            \addplot[blue, dashed, thick, error bars/.cd, y dir=both, y explicit] table[y index = 2, y error minus index = 7, y error plus index = 8]
                {Sparse_norms_iqr.dat};
            \addplot[orange, densely dashdotdotted, thick, error bars/.cd, y dir=both, y explicit] table[y index = 3, y error minus index = 9, y error plus index = 10]
                {Sparse_norms_iqr.dat};
            \addplot[black!30!green, densely dashdotdotted, thick, error bars/.cd, y dir=both, y explicit] table[y index = 4, y error minus index = 11, y error plus index = 12]
                {Sparse_norms_iqr.dat};
            \addplot[only marks, red, mark size =1.5pt,  mark=square*,
                mark options={fill=red, solid}] table[y index = 1]
                {Sparse_norms_iqr.dat};
            \addplot[only marks, blue, dashed, mark size =1.5pt,  mark=*, mark options={fill=blue, solid}] table[y index = 2]
                {Sparse_norms_iqr.dat};
            \addplot[only marks, orange, dashed, mark size =1.5pt,  mark=*, mark options={fill=orange, solid}] table[y index = 3]
                {Sparse_norms_iqr.dat};
            \addplot[only marks, black!30!green, dashed, mark size =1.5pt,  mark=*, mark options={fill=black!30!green, solid}] table[y index = 4]
                {Sparse_norms_iqr.dat};
            \legend{{\small{SSA}}, {\small{RO-L}}, {\small{RO-P}}, {\small{ODE}}}
        \end{axis}
    \end{tikzpicture}
    \caption{Comparison between the singular space approach of this paper (denoted by SSA), the Riemann-Oracle method of \cite{oracle} (denoted by RO-L and RO-P) and the ODE approach of \cite{Sicilia} for sparse matrices of increasing sizes. RO-L denotes the augmented Lagrangian formulation of \cite{oracle}, while RO-P denotes its penalty method formulation. The figure plots the median running time in seconds (above) and the median computed distance to singularity (below) for a sample of 40 randomly generated sparse matrices for each size $n$. The error bars represent the interquartile range corresponding to the 25th and 75th percentiles. For distances, the upper quartiles are almost identical for all methods. There is some variation in the lower quartiles. SSA, RO-L, and RO-P compute almost identical lower quartile distances, while the values for ODE are slightly higher in some cases. } 
    \label{fig:ODE_RO_comparison}
\end{figure}

\subsubsection{Comparison with the non-linear equations approach} \label{sec:comp_newton}
We compare our method against the non-linear equations approach proposed in \cite{poloni2026}. %\footnote{\color{red}This method appeared on arXiv after the submission of this manuscript. Hence, it is not included in the previous experiments; instead, we perform separate additional comparisons with this method.} 
We start with Toeplitz structures, and use the implementation \path{nearest_singular_structured_dense_newtonlike.m} from the repository \href{https://github.com/fph/NearestSingularAsSystem}{\nolinkurl{github.com/fph/NearestSingularAsSystem}}. For our method, we use the augmented Lagrangian approach of Algorithm~\ref{alg:lagrangian}. We generate 40 Toeplitz matrices as before for sizes between $10$ and $150$, and record the computed distances and running times. The results are shown in Figure \ref{fig:newtonlike_comparison}. Both methods yield generally similar distances to singularity. For very small values of $n$, the non-linear equations approach runs faster. However, the method of this paper scales better for larger sizes, and becomes faster at approximately size $n = 50$.

\begin{figure}[h!]
    \begin{tikzpicture}
        \begin{axis}[
                legend style={at={(0.03,0.97)},anchor=north west},
                width=0.5*\linewidth,
                xtick={0, 30, 60, 90, 120, 150},
                height= 6cm,
                title={{\small{Median of time (linear)}}},
                xlabel={size ($n$)},
                ylabel={time (s)}
            ]
            \addplot[red, densely dashdotdotted, thick, error bars/.cd, y dir=both, y explicit] table[y index = 1, y error minus index = 3, y error plus index = 4]
                {Toe_newtonlike_times_iqr_40.dat};
            \addplot[blue, dashed, thick, error bars/.cd, y dir=both, y explicit] table[y index = 2, y error minus index = 5, y error plus index = 6]
                {Toe_newtonlike_times_iqr_40.dat};
            \addplot[only marks, red, mark size =1.5pt,  mark=square*,
                mark options={fill=red, solid}] table[y index = 1]
                {Toe_newtonlike_times_iqr_40.dat};
            \addplot[only marks, blue, dashed, mark size =1.5pt,  mark=*, mark options={fill=blue, solid}] table[y index = 2]
                {Toe_newtonlike_times_iqr_40.dat};
            \legend{{\small{SSA}}, {\small{Newton-like}}}
        \end{axis}
    \end{tikzpicture}
    \begin{tikzpicture}
        \begin{semilogyaxis}[
                legend style={at={(0.97,0.03)},anchor=south east},
                width=0.5*\linewidth,
                xtick={0, 30, 60, 90, 120, 150},
                height= 6cm,
                title={{\small{Median of time (logarithmic)}}},
                xlabel={size ($n$)},
                ylabel={time (s)}
            ]
            \addplot[red, densely dashdotdotted, thick, error bars/.cd, y dir=both, y explicit] table[y index = 1, y error minus index = 3, y error plus index = 4]
                {Toe_newtonlike_times_iqr_40.dat};
            \addplot[blue, dashed, thick, error bars/.cd, y dir=both, y explicit] table[y index = 2, y error minus index = 5, y error plus index = 6]
                {Toe_newtonlike_times_iqr_40.dat};
            \addplot[only marks, red, mark size =1.5pt,  mark=square*,
                mark options={fill=red, solid}] table[y index = 1]
                {Toe_newtonlike_times_iqr_40.dat};
            \addplot[only marks, blue, dashed, mark size =1.5pt,  mark=*, mark options={fill=blue, solid}] table[y index = 2]
                {Toe_newtonlike_times_iqr_40.dat};
            \legend{{\small{SSA}}, {\small{Newton-like}}}
        \end{semilogyaxis}
    \end{tikzpicture}
    \begin{tikzpicture}
        \begin{axis}[
                legend pos = north east,
                legend columns=0,
                width=\linewidth,
                xtick={0, 30, 60, 90, 120, 150},
                height=4cm,
                title={{\small{Median of computed distance}}},
                xlabel={size ($n$)},
                ylabel={distance},
                ymin=-0.1, ymax=1.75,
            ]
            \addplot[red, densely dashdotdotted, thick, error bars/.cd, y dir=both, y explicit] table[y index = 1, y error minus index = 3, y error plus index = 4]
                {Toe_newtonlike_norms_iqr_40.dat};
            \addplot[blue, dashed, thick, error bars/.cd, y dir=both, y explicit] table[y index = 2, y error minus index = 5, y error plus index = 6]
                {Toe_newtonlike_norms_iqr_40.dat};
            \addplot[only marks, red, mark size =1.5pt,  mark=square*,
                mark options={fill=red, solid}] table[y index = 1]
                {Toe_newtonlike_norms_iqr_40.dat};
            \addplot[only marks, blue, dashed, mark size =1.5pt,  mark=*, mark options={fill=blue, solid}] table[y index = 2]
                {Toe_newtonlike_norms_iqr_40.dat};
            \legend{{\small{SSA}}, {\small{Newton-like}}}
        \end{axis}
    \end{tikzpicture}
    \caption{Comparison between the singular space approach of this paper (denoted by SSA) and the nonlinear equations approach of \cite{poloni2026} (denoted by Newton-like) for Toeplitz structured matrices of increasing sizes. The figure plots the median running time in seconds (above) and the median computed distance to singularity (below) for a sample of 40 randomly generated matrices for each size $n \in \{10, 30, 50, 70, 90, 110, 130, 150\}$. The error bars represent the interquartile range corresponding to the 25th and 75th percentiles.}
    \label{fig:newtonlike_comparison}
\end{figure}

Next, we perform a comparison for sparse structures for which we use the method \path{nearest_singular_sparse_newtonlike.m} from the authors' repository. In \cite{poloni2026}, the authors use the matrix \texttt{orani678} from SuiteSparse Matrix Collection as a numerical example, which is a sparse matrix of size $2,529$ with $90,158$ non-zeros. For this matrix, we compute the distance to singularity preserving the sparsity of the matrix. The non-linear equations approach of \cite{poloni2026} and the SSA proposed in this paper both compute a distance of approximately $0.0268$. However, the non-linear equations approach finishes in $3.88$ seconds, whereas the SSA (augmented Lagrangian formulation of Algorithm \ref{alg:lagrangian}) requires $25.02$ seconds. 

The \texttt{orani678} matrix is significantly larger than those used in the previous experiment for sparse structures in Figure~\ref{fig:ODE_RO_comparison}, where we used the Tikhonov regularization approach of Algorithm~\ref{alg:reg}. For this matrix, the augmented Lagrangian approach is faster: the Tikhonov regularization approach takes more than $60$ seconds to run. One possible explanation is that, for this larger matrix, the more favourable scaling of the augmented Lagrangian approach outweighs the cost of its additional numerical optimization step, consistent with the data for Toeplitz matrices shown in Figure~\ref{fig:lagrangian_comparison}. %Altogether, these results imply that the method of this paper does not scale as competitively for large problems with sparse structures as the non-linear equations approach of \cite{poloni2026}. 

% There are cases where large and sparse matrices  polynomial eigenvalue problems would have an additional block-Toeplitz structure (see, e.g., \cite[Section 6]{oracle}). 

\subsubsection{Matrix polynomial from NLEVP} \label{sec:NLEVP}
The distance to singularity problem for polynomial matrices can be translated into a structured distance to singularity problem for scalar matrices; see, for example, \cite[Section 6]{oracle} for a detailed explanation. Typically, the resulting matrix is rectangular. Although the method of this paper was developed for square matrices, strategy (i) of Table~\ref{table:compare_alternate} extends directly to matrices with more rows than columns. %can be applied to tall and thin rectangular matrices in a straightforward manner.

In this numerical experiment, we apply the method of this paper to the \path{mobile_manipulator} problem of the NLEVP collection \cite{nlevp}, which is a $5 \times 5$ quadratic eigenvalue problem associated with the system
\begin{align*}
    M\ddot{x}(t) + D\dot{x}(t) + Kx(t) & = Bu(t), \\
    y(t)                               & = Cx(t),
\end{align*}
where the coefficient matrices are
\begin{align*}
    M = \begin{bmatrix} M_0 & 0 \\ 0 & 0 \end{bmatrix}, \quad D = \begin{bmatrix} D_0 & 0 \\ 0 & 0 \end{bmatrix}, \quad K = \begin{bmatrix} K_0 & -F_0^T \\ F_0 & 0 \end{bmatrix},
\end{align*}
with
\begin{align*}
    M_0 & = \begin{bmatrix} 18.7532 & -7.94493 & 7.94494 \\ -7.94493 & 31.8182 & -26.8182 \\ 7.94494 & -26.8182 & 26.8182 \end{bmatrix},  \\
    D_0 & = \begin{bmatrix} -1.52143 & -1.55168 & 1.55168 \\ 3.22064 & 3.28467 & -3.28467 \\ -3.22064 & -3.28467 & 3.28467 \end{bmatrix}, \\
    K_0 & = \begin{bmatrix} 67.4894 & 69.2393 & -69.2393 \\ 69.8124 & 1.68624 & -1.68617 \\ -69.8123 & -1.68617 & -68.2707 \end{bmatrix}, \\
    F_0 & = \begin{bmatrix} 1 & 0 & 0 \\ 0 & 0 & 1 \end{bmatrix}.
\end{align*}
This problem arises from modelling a two-dimensional three-link mobile manipulator as a time-invariant descriptor control system.

% The quadratic $Q(\lambda) = \lambda^2 M + \lambda D + K$ is close to being nonregular [Byers et al. 1998; Higham and Tisseur 2002]. 

We compute the distance to singularity preserving the sparsity structure of the problem.
%This requires testing both the right and left kernels for $d=4$ by transposing the polynomial coefficients, and constructing a $35 \times 25$ Toeplitz block matrix for each.
This is done for the Tikhonov regularized approach of this paper as well as the Riemann-Oracle method. We could not include the nonlinear equations approach~\cite{poloni2026} in this comparison because its implementation terminated at the first iteration on this example, apparently due to a rank-deficient Jacobian.
% Unfortunately, we could not compare against the non-linear system approach \cite{poloni2026}, as it fails for this example: it is not rigorously implemented for rectangular matrices. %As such, we perform this comparison only with RO.
We perform the experiment using the seed \texttt{rng(1)}. The source code for running this experiment can be found in the \path{experiment_mobile_manipulator.m} file of the GitHub repository. The method of this paper converges to a relative distance of $0.00555515$, requiring $1.13$ seconds. In contrast, Riemann-Oracle computes a relative distance $0.00555477$, requiring $29.31$ seconds. These results are presented in Table \ref{table:nlevp}. %This experiment demonstrates the massive computational advantage and scalability of utilizing singular spaces directly instead of operating over fully dense manifold formulations.

\begin{table}[h!]
    \caption{Comparison of the structured distance to singularity for the \texttt{mobile\_manipulator} problem from the NLEVP collection. The distances are relative, scaled by the Frobenius norm of the polynomial coefficient matrices. The table also records the running times.}
    \label{table:nlevp}
    \begin{center}
        \begin{tabular}{|c||c|c|}
            \hline
            Method           & Computed distance & Running time (s) \\
            \hline
            SSA (This paper) & 0.00555515        & 1.13             \\
            \hline
            Riemann-Oracle   & 0.00555477        & 29.31            \\
            \hline
        \end{tabular}
    \end{center}
\end{table}

% In this section, we discuss reasons for why the SSA performs well computationally. These include a favourable coupling of variable blocks for an alternating minimization approach, the reliance on Level 3 BLAS for the most expensive operations, as well as the low growth-rate in iteration count for larger input sizes. 

% In this section, we discuss the computational properties of the SSA, including the cost per iteration, the coupling between the variable blocks in the alternating minimization, and the observed dependence of the iteration count on the input size.

% \subsection*{Overall lightness}
% Algorithm \ref{alg:reg} alternates between minimizing the objective function \eqref{eq:obj_g} over the perturbation $\delta$ and minimizing it over the singular space $S$. Both of these operations are rather light, especially for strategy (i) of Table \ref{table:compare_alternate}, which uses spaces of type $\mathcal{S}_v$. Indeed, for spaces of type $\mathcal{S}_v$, minimizing over $\delta$ has a cheap, closed-form solution in \eqref{deltastar}, and finding the next singular space $S$ in Procedure \ref{proc:construct_reg} only requires the computation of the smallest singular triplet of a matrix the same size as the input.

\section{Asymptotic complexity}
Let $n$ denote the size of the input matrix, and let $p$ denote the dimension of the structure space. To simplify the complexity analysis, we assume that $p \geq n$, which holds for many common structures. For the SSA, the most expensive operation asymptotically is forming the product $M M^*$ in \eqref{deltastar}, which is $O(n^2 p)$ for approach (i) of Table \ref{table:compare_alternate}. %Note that the computation of the smallest singular triplet of an $n \times n$ matrix is asymptotically cheaper. % has complexity $O(n^2)$. 
The authors of Riemann-Oracle \cite{oracle} note that the cost per iteration of their method is $O(mp(n+\min(m, p))$), which for square matrices ($m=n$) and under the assumption $p \geq n$ becomes $O(n^2 p)$. %While for Riemann-Oracle, the cost $O(n^2 p)$ comes from the computation of an \texttt{svd} for a matrix of size $n \times p$, %which is significantly more expensive than a single matrix-matrix product (which relies on highly optimized Level 3 BLAS), 
As such, the per-iteration asymptotic cost for both algorithms is equal. The total asymptotic cost is then $O(n^2 p)$ multiplied by the number of iterations needed.

Figure~\ref{fig:lagrangian_comparison} shows that, for the SSA of this paper, the total amount of inner iterations remains somewhat constant throughout as the size of the input matrices increases. This contributes to the favourable scaling of the running times observed in the experiments. %As such, the reason why the SSA scales well is the fact that the total number of iterations needed grows at a low rate.

\section{Conclusion and future work} \label{sec:conclusion}
In this paper, we developed a framework for finding the structured distance to singularity based on singular spaces. The framework leads to a nonregularized method based on successive least-squares problems, a Tikhonov regularized formulation, and an augmented Lagrangian variant. We established results on accumulation points of the regularized iterations and proved an optimality property for singular spaces defined by rectangular zero blocks. Numerical experiments showed substantial reductions in running time for many of the tested problems, while the computed distances remained comparable to those obtained by existing methods.

% We proved an optimality property for singular spaces defined by rectangular zero blocks, which motivated their use in the algorithm. In the Tikhonov regularized formulation, the resulting method alternately fixes one variable of the objective function, and finds a global minimizer for the other argument. %The nonlinear equations approach of \cite{poloni2026} was faster for the large sparse example considered. %Even for sparse matrices, there are cases where the method of this paper offers an improvement: For example, sparse polynomial eigenvalue problems have an additional block-Toeplitz structure (see, e.g., \cite[Section 6]{oracle}), and hence the overall structure is not strictly sparse. %Moreover, the method of this paper seems to extend to the rectangular case in a somewhat straightforward manner, as numerically demonstrated in Section~\ref{sec:NLEVP}. While the authors of \cite{poloni2026} have performed numerical experiments for certain rectangular matrices, their method fails for the rectangular matrix in Section~\ref{sec:NLEVP}. Moreover, the theory \cite{poloni2026} is only developed for square matrices. As such, the method proposed in this paper can offer a speedup compared to existing matrices for sparse rectangular matrices. }

% There are cases where large and sparse matrices polynomial eigenvalue problems would have an additional block-Toeplitz structure (see, e.g., \cite[Section 6]{oracle}). 

One future research direction is to investigate how the singular-space framework of this paper could be adapted to other matrix nearness problems, including the nearest unstable matrix problem and the approximate GCD problem. Another possible extension of this work would be to more general structured low-rank approximation problems. %In the same spirit as in this paper, one could attempt to characterize linear subsets inside the low-rank manifold, and base a numerical method on this insight. %We leave these research directions for future work.
A similar formulation in terms of linear subspaces of matrices with a prescribed rank bound could provide a framework for comparing different choices of search spaces and analysing the resulting methods.

\section*{Acknowledgements} We sincerely thank Nicola Guglielmi, Christian Lubich and Stefano Sicilia for providing the MATLAB codes for their algorithm in \cite{Sicilia}. %\cite[Section V.8]{guglielmi_book}. 
We also thank Ethan N. Epperly, Yuji Nakatsukasa and Taejun Park for kindly providing the function \texttt{svdmin.m}. %{\color{red}Lastly, we thank Konstantin Usevich and two anonymous reviewers, whose insightful comments helped improve the presentation of the paper.}

\section*{Declaration of generative AI and AI-assisted technologies in the manuscript preparation process}

During the preparation of this work, the author used ChatGPT to assist with drafting and revising text, improving clarity, and organizing the presentation. The author reviewed and edited the resulting content and takes full responsibility for the content of this article.

\section*{Declaration of competing interest}
The author has no competing interests to declare.

\bibliographystyle{abbrv}
\bibliography{references}

\appendix

\section{Compactness of \texorpdfstring{$\mathfrak{S}_n(\F)$}{S\_n(F)}} \label{appendix:compact} We show that $\mathfrak{S}_n(\F)$ is compact in a natural way. Clearly, $$\mathfrak{S}_n(\F) \subset \bigsqcup_{k=0}^{n^2} \mbox{Gr}_k(\F^{n \times n}),$$ where $\mbox{Gr}_k$ denotes the Grassmann manifold of $k$-dimensional linear subspaces. Let $d_k$ denote the standard metric on $\mbox{Gr}_k$. We can equip the set $\bigsqcup_{k=0}^{n^2} \mbox{Gr}_k(\F^{n \times n})$ with the metric $d$ defined as $d(x,y) = \min(1, d_k(x,y))$, if $x,y \in \mbox{Gr}_k(\F^{n \times n})$, and $d(x,y) = 1$ otherwise. As each $\mbox{Gr}_k$ is compact, it holds that their disjoint union $\bigsqcup_{k=0}^{n^2} \mbox{Gr}_k(\F^{n \times n})$ is sequentially compact with respect to the metric $d$: by Pigeonhole principle, for any infinite sequence $(x_i)_{i \in \Z^+}$, there exists an index $j$ such that $(x_i)_{i \in \Z^+}$ has an infinite amount of terms in $\mbox{Gr}_j$. As $\mbox{Gr}_j$ is compact with respect to the metric $d$, there exists a convergent subsequence of $(x_i)_{i \in \Z^+}$. This proves sequential compactness.

To prove that $\mathfrak{S}_n(\F)$ is compact, we note that every closed subset of a compact set is compact. To see that $\mathfrak{S}_n(\F)$ is closed, we note that any convergent sequence in $\mbox{Gr}_k$ has a convergent sequence of representatives in the Stiefel manifold $V_k(\F^{n^2})$. The set 
\[
    \{ [B_1 \ \cdots \ B_k ] \in \F^{n \times kn} : \ \det\Big(\sum_{i=1}^k B_i x_i\Big) \equiv 0\}
\]
is Zariski closed, and hence its intersection with $V_k(\F^{n^2})$ is closed. It follows that $\mathfrak{S}_n(\F)$ contains its limit points, and is hence closed.

\section{Proof of Example~\ref{ex:counter}}
\label{appendix:counter}

\begin{proof}
    Let $M(x,y,z) = M_x x + M_y y+ M_z z$, where
    \begin{align*}
        M_x =
        \begin{bmatrix}
            0 & 1 & 0 \\
            0 & 0 & 0 \\
            0 & 0 & 1
        \end{bmatrix}, \
        M_y =
        \begin{bmatrix}
            1 & 0 & 0 \\
            0 & 0 & 1 \\
            0 & 0 & 0
        \end{bmatrix}, \
        M_z =
        \begin{bmatrix}
            0  & 0 & 0 \\
            0  & 1 & 0 \\
            -1 & 0 & 0
        \end{bmatrix} .
    \end{align*}
    First assume that the result holds for either $|I| = 3$ or $|J| = 3$. Then, there exist $A, B \in \mbox{GL}(3, \F)$ and a nonzero $v \in \F^3$ such that either $A^{-1}M(x,y,z)B^{-1} v = 0$ or $(A^{-1}M(x,y,z)B^{-1})^T v = 0 \ \forall x,y,z$. Note that each $M_x, M_y, M_z$ are rank-2, with right kernels spanned by $e_1, e_2, e_3$, respectively, and left kernels spanned by $e_2, e_3, e_1$, respectively. These vectors remain linearly independent after a change of basis, and as $v$ needs to belong to the intersection of these kernels (after a change of basis by $A$ or $B$), it follows that $v=0$, which is a contradiction.

    Let us then assume that the result holds with $|I| = 2$ and $|J| = 2$. It then follows that there exist full rank matrices $A,B \in \F^{2\times 3}$ such that $A M(x,y,z) B^T = 0 \ \forall x,y,z$. This implies that $A M_x B^T = A M_y B^T = A M_z B^T = 0$. For this to hold, it must be that $\text{im}(B^T) \subset \mbox{ker}(A M_x) \cap \mbox{ker}(A M_y) \cap \mbox{ker}(A M_z)$, which can only hold if $\mbox{ker}(A M_x) = \mbox{ker}(A M_y) = \mbox{ker}(A M_z)$, in which case $e_1, e_2, e_3 \in \mbox{ker}(A M_x)$. This is impossible as $\mbox{dim} \ \mbox{ker}(A M_x) \leq 2$.
\end{proof}

\section{Proof of Proposition~\ref{prop:dimensions_missing}}
\label{appendix:dimensions_missing}

\begin{proof}
    Let $S \subset \R^{n\times n}$ denote a vector space of singular matrices such that $M \in S$ and $k := \rank M$ is the highest rank of any element in $S$. Then,
    \begin{align*}
        M = [U_1 \dots U_{k} ] [V_1 \dots V_{k} ]^T,
    \end{align*}
    for some orthonormal $\{U_i\}_{1\leq i \leq k}$ and orthogonal $\{V_i\}_{1\leq i \leq k}$. Let $\{U_i\}_{k+1 \leq i \leq n}$ and \linebreak $\{V_i\}_{k+1 \leq i \leq n}$ be such that $U := [U_1 \dots U_n]$ is an orthonormal matrix and $[V_1 \dots V_n]$ is an orthogonal matrix.

    We consider the following orthogonal basis for an $n$ dimensional subspace: $\{U_n V_i^T \linebreak + U_i V_n^T \}_{i=1 \dots n}.$ It is straightforward to check that this set of vectors is orthogonal in the Frobenius inner product. Indeed,
    \begin{align*}
         & \langle U_n V_i^T + U_i V_n^T, \  U_n V_j^T + U_j V_n^T \rangle_F
        \\ &= \tr(V_i V_j^T + \delta_{in} V_n V_j^T + \delta_{nj} V_i V_n^T + \delta_{ij} V_n V_n^T)
        \\ &= \delta_{ij}\|V_i\|^2 + 2 \delta_{in} \delta_{nj}\|V_n\|^2 + \delta_{ij} \|V_n\|^2,
        %&= \delta_{ij}(\|V_i\|^2 + 2 \delta_{in}\|V_n\|^2 + \|V_n\|^2),
    \end{align*}
    which is equal to zero if and only if $i \neq j$.

    Next, we will see what happens when we add an arbitrary element of this $n$-dimensional subspace to $M$:
    \begin{align*}
         & M + \sum_{i=1}^n \alpha_i (U_n V_i^T + U_i V_n^T)
        \\ &= \left[U_1 \ \ \cdots \ \ U_{k} \ \ U_n \ \ \sum_{i=1}^n \alpha_i U_i \right] \left[V_1 \ \ \cdots \ \ V_{k} \ \ \sum_{i=1}^n {\alpha}_i V_i \ \ V_n\right]^T.
    \end{align*}
    Clearly, both matrices in the product have full rank $\min \{k+2,n\}$ if any $\alpha_{k+1}, \dots, \alpha_{n-1}$ is nonzero. Otherwise, the product becomes
    \begin{align*}
         & \left[U_1 \ \ \cdots \ \ U_{k} \ \ U_n \ \ \alpha_n U_n + \sum_{i=1}^k \alpha_i U_i \right] \left[V_1 \ \ \cdots \ \ V_{k} \ \ {\alpha}_n V_n + \sum_{i=1}^k {\alpha}_i V_i \ \ V_n\right]^T
        \\ &= \left[U_1 \ \ \cdots \ \ U_{k} \ \ U_n \right]
        \left[V_1 + {\alpha}_1 V_n \ \ \cdots \ \ V_{k} + {\alpha}_{k} V_n \ \ 2 {\alpha}_n V_n + \sum_{i=1}^k {\alpha}_i V_i \right]^T.
    \end{align*}
    We see that the matrix on the left has orthonormal columns and is hence full rank $k+1$.  After applying elementary row operations to the matrix on the right hand side, it becomes
    \begin{align*}
        \left[V_1 + {\alpha}_1 V_n \ \ \cdots \ \ V_{k} + {\alpha}_{k} V_n \ \ (2 {\alpha}_n - \sum_{i=1}^{k} {\alpha}^2_i) V_n \right]^T.
    \end{align*}
    As $\{V_i\}_{1\leq i \leq n}$ are orthogonal, it is clear that the matrix is rank-deficient if and only if
    \begin{equation*}
        \begin{aligned}
             & 2 {\alpha}_n - \sum_{i=1}^{k} {\alpha}^2_i = 0.
        \end{aligned}
    \end{equation*}
    In other words, an element $\widehat M$ of the $n$-dimensional affine subspace $A := M + \mbox{span}\{U_n V_i^T \linebreak + U_i V_n^T \}_{i=1 \dots n}$ satisfies $\mbox{rank}(\widehat M) \leq k$ if and only if its coefficients satisfy
    \begin{equation}
        \begin{aligned} \label{eq:coeff_constraint}
              & 2 {\alpha}_n - \sum_{i=1}^{k} {\alpha}^2_i = 0, \\
            [ & \alpha_{k+1} \ \cdots \ \alpha_{n-1}] = 0.
        \end{aligned}
    \end{equation}
    On the other hand, the intersection of the affine space $A$ with any vector subspace of $\R^{n\times n}$ should be an affine subspace of $\R^{n\times n}$. The highest dimension for an affine subspace $B \subset A$ with $M \in B$ satisfying the nonlinear constraint \eqref{eq:coeff_constraint} is clearly zero. This can be seen by considering the one-dimensional affine subspace $\left\{M + t \sum_{i=1}^n \beta_i (U_n V_i^T + U_i V_n^T)\right\}_{t\in \R}$ and noting that
    \begin{align*}
        2 {\beta}_n t - \sum_{i=1}^{k} {\beta}^2_i t^2 = 0
    \end{align*}
    has a finite amount of solutions for all $[\beta_1 \ \cdots \ \beta_k \ \beta_n] \neq 0$. As matrices in $A$ with rank less than or equal to $k$ need to satisfy \eqref{eq:coeff_constraint}, it follows that the intersection $A \cap S$ is zero dimensional.
\end{proof}

\end{document}